\documentclass[11pt,final]{article}
\usepackage{amsmath,amsthm,amsfonts,amssymb,graphicx,hyperref,enumerate,psfrag}
\usepackage[a4paper,margin=2.5cm]{geometry}
\usepackage{fullpage}
\usepackage{enumitem}
\usepackage{tikz}
\usetikzlibrary{positioning}
\tikzset{main node/.style={circle,fill=blue!20,draw,minimum size=1cm,inner sep=0pt},
            }
\tikzstyle{every loop}=[]

\usepackage[utf8]{inputenc}

\newtheorem{lemma}{Lemma}

\newtheorem{remark}{Remark}

\newtheorem{proposition}[lemma]{Proposition}
\newtheorem{theorem}[lemma]{Theorem}

\newtheorem{corollary}[lemma]{Corollary}

\newcommand{\dE}{\mathbb {E}}
\newcommand{\dP}{\mathbb{P}}
\newcommand{\dZ}{\mathbb {Z}}
\newcommand{\dN}{\mathbb {N}}
\newcommand{\dR}{\mathbb {R}}
\newcommand{\dC}{\mathbb {C}}

\newcommand{\cB}{\mathcal {B}}

\newcommand{\cG}{\mathcal {G}}

\newcommand{\cN}{\mathcal {N}}

\newcommand{\cR}{\mathcal {R}}

\newcommand{\cT}{\mathcal {T}}

\newcommand{\cX}{\mathcal {X}}
\newcommand{\cPT}{P_{\mathcal {T}}}

\newcommand{\e}{\overrightarrow{e}}
\newcommand{\E}{\overrightarrow{E}}
\newcommand{\kT}{\mathfrak{T}}

\newcommand{\We}{W^{(\epsilon)}}

\title{Cutoff for random lifts of weighted graphs}
\author{Guillaume Conchon-\hspace{0.3mm}-Kerjan \thanks{LPSM UMR 8001, Université Paris Diderot, Université de Paris, e-mail: gconchon[at]lpsm.paris}}
\begin{document}
\maketitle
\begin{abstract}
We prove a cutoff for the random walk on random $n$-lifts of finite weighted graphs, even when the random walk on the base graph $\cG$ of the lift is not reversible. The mixing time is w.h.p. $t_{mix}=h^{-1}\log n$, where $h$ is a constant associated to $\cG$, namely the entropy of its universal cover. Moreover, this mixing time is the smallest possible among all $n$-lifts of $\cG$. In the particular case where the base graph is a vertex with $d/2$ loops, $d$ even, we obtain a cutoff for a $d$-regular random graph (as did Lubetzky and Sly in \cite{cutoffregular} with a slightly different distribution on $d$-regular graphs, but the mixing time is the same). 
\end{abstract}

\section{Introduction}
\subsection{The cutoff phenomenon}
The way random walks converge to equilibrium on a graph is closely related to essential geometrical properties of the latter (such as the typical distance between vertices, its diameter, its expansion, the presence of traps or bottlenecks, etc.), giving an important motivation for studying mixing times. 
\\
For a Markov chain on a discrete state space $\Omega$ and transition matrix $P$ that admits an invariant distribution $\pi$, the \textbf{$\varepsilon$-mixing time from $x$} is
$$t_{x}(\varepsilon):=\inf\{t\geq 0 \, \vert \, \Vert P^t(x,\cdot)-\pi \Vert_{TV}\leq \varepsilon\},$$
where $\Vert \nu_1- \nu_2\Vert_{TV}:=\sup_{S\subset \Omega} \left(\nu_1(S)-\nu_2(S)\right)$ is the total variation distance between the probability measures $\nu_1$ and $\nu_2$ on $\Omega$. The \textbf{worst-case mixing time} $t_{max}(\varepsilon):=\sup\{t_{x}(\varepsilon), \, x\in \Omega\}$ is often the quantity of main interest. Other distances than the total variation distance might be considered. A straightforward computation shows that $t\mapsto \Vert P^t(x,\cdot)-\pi \Vert_{TV}$ is a non-increasing function, so that the definition of mixing time is relevant.
\\
For a sequence of Markov chains $(\Omega_n,P_n,\pi_n)_{n\geq 0}$, there is \textbf{cutoff} when for all $\varepsilon, \varepsilon' \in (0,1)$, $t_{max}^{(n)}(\varepsilon)/t_{max}^{(n)}(\varepsilon ')\rightarrow 1$ as $n\rightarrow +\infty$. 
\\
While the cutoff phenomenon remains far from being completely understood, first examples of it were given in the 1980s for different random walks on finite groups (see \cite{DiacoShah} or\cite{aldousD} on the symmetric group) or on spaces that can be "factorized" as a $n$-product of a base space (such as $\dZ_2^n$ in \cite{Aldous}), and this direction is still investigated nowadays (see for instance \cite{HermonThomas} on random Cayley graphs of abelian groups).
\\
On the other hand, a class of graphs where random walks mix fast and where cutoff is expected are the expander graphs. These are sequences $(G_n)_{n\geq 0}$ of graphs whose size goes to infinity (say $G_n$ has $n$ vertices) and whose isoperimetric constant is bounded away from $0$: there exists $c>0$ independent of $n$ such that for any subset $S$ of at most $n/2$ vertices of $G_n$, $\vert \partial S\vert \geq c \vert S \vert$, where $\partial S \subset S^c$ is the set of vertices adjacent to vertices of $S$. The book of Lubotzky \cite{lubotzky} gives a good panoramic view about the search for such graphs in the 1980s. This expansion property entails indeed the existence of a spectral gap (this implication is called the "Cheeger bound", see \cite{alonmilman} for instance): the second largest eigenvalue of the transition matrix $P_n$ of the SRW on $G_n$ is bounded away from the largest one as $n\rightarrow +\infty$. It is classical that this spectral gap implies in turn that the SRW on $G_n$ mixes in $O(\log n)$ steps.
\\
The simplest expander model is the random $d$-regular graph (i.e. $G_n(d)$ is chosen uniformly among graphs of $n$ vertices having all degree $d$). Friedman \cite{friedmanlifts} proved in 2002 that w.h.p., $G_n(d)$ almost achieves the largest possible spectral gap, while Lubetzky and Sly \cite{cutoffregular} proved in 2008 that the SRW and the NBRW (Non-Backtracking Random Walk, i.e. a SRW conditioned at each step on not going back along the edge it has just crossed) on $G_n(d)$ admits a cutoff. Several papers followed on cutoffs for other sparse graphs: see for instance \cite{blps} for the SRW on the largest component of a supercritical Erdös-Rényi random graph, or \cite{nbrw} for the NBRW on a configuration model. 
\\
Very recently, there has been increasing interest in mixing times on dynamical graphs (typically, edges are re-sampled at random at a given rate), when the mixing time profile is already well-known on a static version of the graph (for instance \cite{AGuldas}, \cite{SousiThomas} and \cite{CaputoQuattro}). 
\\
A natural way of combining the "product of base space" and the "expanding sparse graph" perspectives for cutoff is to consider random walks on random $n$-lifts of a fixed graph $\cG$.
 
\subsection{Random walks on weighted graphs}
For a multigraph $\cG$ (hence we allow multiple edges and multiple loops), denote $V_{\cG}$ its \textbf{vertex set} and $E_{\cG}$ its \textbf{edge set}. Every edge $e\in E_{\cG}$ gives rise to two opposite oriented edges. Denote $\E_{\cG}$ the set of oriented edges of $\cG$. For each $\e \in \E_{\cG}$, note $\e^{-1}$ its opposite. We study weighted random walks by giving to each $\e\in \E_{\cG}$ a nonnegative weight $w(\e)$, so that
\begin{itemize}
\item every $e\in E_{\cG}$ has at least one orientation with positive weight, 

\item for all $u\in V_{\cG}$, the sum of the weights of the oriented edges going out of $u$ is $1$.
\end{itemize} 
\noindent
We define \textbf{the random walk (RW) on $\cG$} as a Markov Chain $(X_t)_{t\geq 0}$ on $V_{\cG}$ with transition matrix $P_{\cG}$ such that for all $u,v\in V_{\cG}$,
$$P_{\cG}(u,v)=\frac{1}{2}\textbf{1}_{\{u=v\}}+\frac{1}{2}\sum_{\e\,: u\rightarrow v}w(\e),$$ 
where "$\e\,: u\rightarrow v$" means that the \textbf{initial vertex} of $\e$ is $u$ and its \textbf{end vertex} is $v$.

\subsection{Random lifts}
Fix now a finite multigraph $\cG$. A \textbf{$n$-lift of $\cG$} is a graph $\cG_n$ with vertex set $V_{\cG_n}:=V_{\cG}\times [n]$, and edge set $E_{\cG_n}$ as follows: fix for each $e\in E_{\cG}$ an arbitrary ordering $(u,v)$ of its endpoints and a permutation $\sigma_e\in \mathfrak{S}_n$,  and draw the edges $\{u_i,v_{\sigma_e(i)}\}$ for all $1\leq i\leq n$ (see Figure 1). Say that $u$ (resp. $v$) is the \textbf{type} of $u_i$ (resp. $v_{\sigma(i)}$) and that $e$ is the \textbf{type} of $\{u_i, v_{\sigma(i)} \}$. 
\\
When the $\sigma_e$'s are uniform independent permutations, $\cG_n$ is a \textbf{random $n$-lift of $\cG$}.
\\
For simplicity of the notations, write $V_n, E_n$ and $\E_n$ for the vertex set, edge set and oriented edge set of $\cG_n$. 
Define as previously the \textbf{type} of an element of $\E_n$ as the corresponding oriented edge of $\E_{\cG}$, and give to each $\e\in \E_n$ the weight of its type.

\begin{tikzpicture}[scale=0.7]

    \node[main node] (1) {$u$};
    \node[main node] (2) [above right = 0.5cm and 3.5cm of 1]  {$v$};
    \node[main node] (3) [below right = 0.5cm and 3.5cm of 1] {$x$};

  \path[draw,thick]
  (1) edge[loop right, min distance = 5mm, in = 130, out = 230, looseness =10, color= blue] node {} (1)
  (1) edge[green] node {} (2)
  (2) edge[bend right, red] node {} (3)
  (2) edge[bend left] node {} (3)
  (3) edge[purple] node {} (1);
\draw [->] (-2.2,-0.5) -- (-2.2,0.5);
\draw (-2.2,0) node[right] {$0.3$};
\draw [<-] (-2.75,-0.5) -- (-2.75,0.5);
\draw (-2.75,0) node[left] {$0.2$};
\draw [->] (1,0.5) -- (2,0.8);
\draw (1.5,0.65) node[above] {$0.4$};
\draw [->] (1,-0.5) -- (2,-0.8);
\draw (1.5, -0.69) node[below] {$0.1$};
    \begin{scope}[xshift=10cm]
    \node[main node] (1) {$u_1$};
    \node[main node] (2) [above right = 0.5cm and 3.5cm of 1]  {$v_1$};
    \node[main node] (3) [below right = 0.5cm and 3.5cm of 1] {$x_1$};
    \node[main node] (4) [below = 3cm  of 1] {$u_2$};
    \node[main node] (5) [above right = 0.5cm and 3.5cm of 4]  {$v_2$};
    \node[main node] (6) [below right = 0.5cm and 3.5cm of 4] {$x_2$};
        \node[main node] (7) [below = 3cm  of 4] {$u_3$};
    \node[main node] (8) [above right = 0.5cm and 3.5cm of 7]  {$v_3$};
    \node[main node] (9) [below right = 0.5cm and 3.5cm of 7] {$x_3$};
\draw [->] (0.2,-1.9) -- (0.2,-2.9);
\draw (0.2,-2.4) node[right] {$0.2$};
\draw [<-] (0.2,-3.4) -- (0.2,-4.4);
\draw (0.2,-3.9) node[right] {$0.3$};
\draw [->] (0.2,-7.1) -- (0.2,-8.1);
\draw (0.2,-7.6) node[right] {$0.2$};
\draw [<-] (0.2,-8.4) -- (0.2,-9.4);
\draw (0.2,-8.9) node[right] {$0.3$};
\draw [->] (-1.1,-1.1) -- (-1.6,-2.1);
\draw (-1.4,-1.6) node[left] {$0.3$};
\draw [->] (-1.2,-10.1) -- (-1.7,-9.1);
\draw (-1.5,-9.6) node[left] {$0.2$};
\draw [->] (0.9,-0.88) -- (1.6,-1.4);
\draw (0.4,-1.6) node[right] {$0.4$};
\draw [->] (0.9,0) -- (1.7,-0.2);
\draw (0.9,0.3) node[right] {$0.1$};
\draw [->] (0.8,-11.4) -- (1.7,-11.2);
\draw (0.7,-11.7) node[right] {$0.4$};
\draw [<-] (1.2,-10.4) -- (0.5,-10.9);
\draw (0.2,-10.1) node[right] {$0.1$};
\draw [->] (0.8,-5.1) -- (1.6,-4.3);
\draw (1.9,-5.2) node[left] {$0.4$};
\draw [->] (0.8,-6.2) -- (1.6,-7.1);
\draw (2.1,-6.3) node[left] {$0.1$};

    \path[draw,thick]
   (1) edge[bend right, blue ] node {} (7)
   (1) edge[purple] node {} (3)
   (1) edge[green] node {} (5)
   (1) edge[blue] node {} (4)
   (2) edge[green] node {} (4)
   (2) edge[bend left] node {} (6)
   (2) edge[red] node {} (3)
   (3) edge node {} (5)
   (4) edge[blue] node {} (7)
   (4) edge[purple] node {} (9)
   (5) edge[red] (6)
   (6) edge[purple] node {} (7)
   (7) edge[green] node {} (8)
   (8) edge[bend right, red] node {} (9)
   (8) edge[bend left] node {} (9)
   ;
   \end{scope}

\end{tikzpicture}

\begin{center}
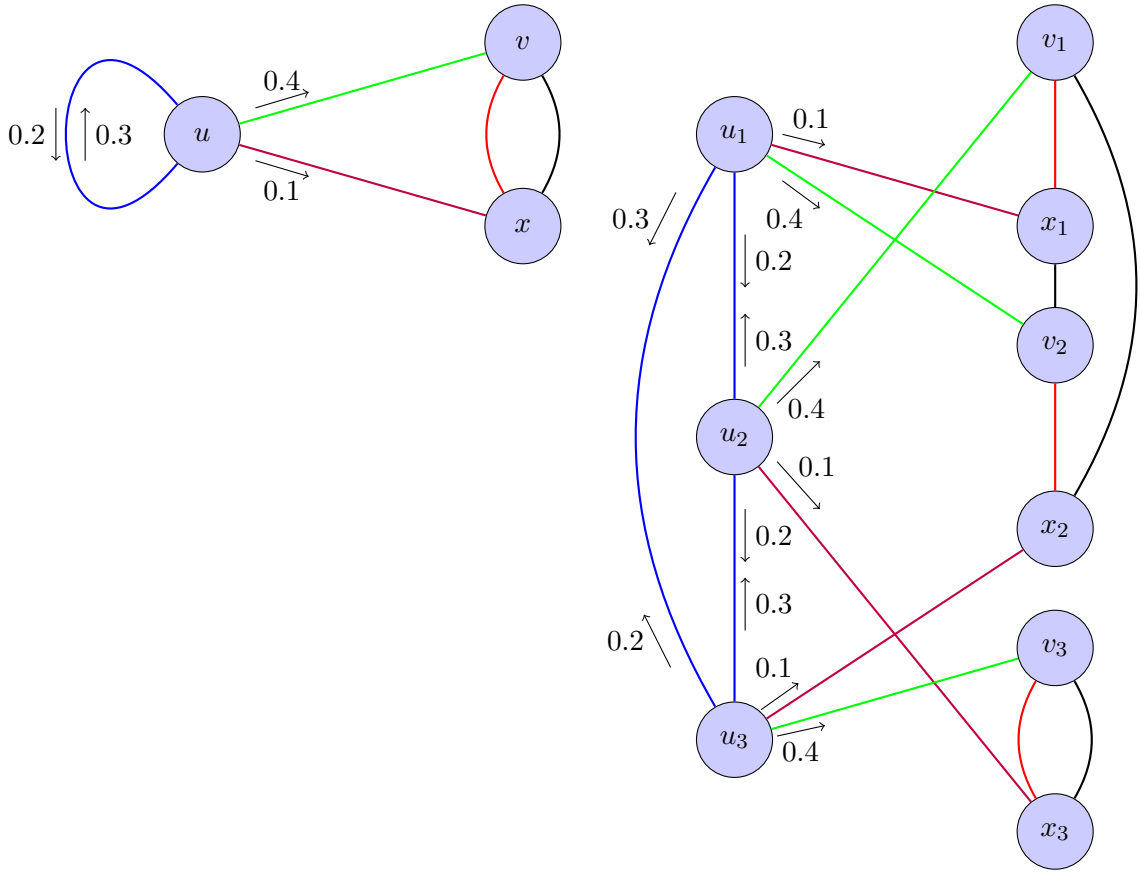
{Figure 1: a weighted graph $\cG$ and a $3$-lift of $\cG$ (not all weights are written on the picture). $\sigma_{\{u,u\}}=(2\, 3\, 1)$, $\sigma_{\{u,v\}}=(2\, 1\, 3)$, $\sigma_{\{u,x\}}=(1\, 3\, 2)$, $\sigma_{\{v,x, \text{red}\}}=(1\, 2\, 3)$,  $\sigma_{\{v,x, \text{black}\}}=(2\, 1\, 3)$}.
\end{center} 
\noindent
Denote $\pi$ the invariant measure of the RW on $\cG$ (if it exists). Note that the RW on $\cG_n$ has an invariant measure $\pi_n$ such that $\pi_n((x,i))=\pi(x)/n$ for all $x\in V_{\cG}$, $i\in [n]$.
\\
\\
The graph structure of random lifts has been studied since the early 2000s (see \cite{liniallifts1}, \cite{liniallifts2}, \cite{liniallifts3} and \cite{liniallifts4}). In particular, it is proved in \cite{liniallifts2} that random $n$-lifts are expanders w.h.p. as $n$ goes to infinity, as long as $\cG$ has at least two cycles.
More recently, spectral properties of lifts have been investigated (see for instance \cite{AddarGrif}, \cite{LinialPuder}, \cite{LubetzkySudakov}): Bordenave \cite{friedmanlifts} generalized Friedman's theorem to the NBRW on random $n$-lifts of a finite graph, then Bordenave and Collins \cite{BordeCol} established a similar result for the SRW. Bordenave and Lacoin \cite{BordeLac} proved that if the RW associated to $\cG$ is reversible, and if the invariant measure is uniform, then the RW on $\cG_n$ admits a cutoff, with a mixing time in $h^{-1}\log n +o(\log n)$ steps, for some constant $h$ (the "entropy") depending on $\cG$. 

\subsection{Results}
We characterize all irreducible graphs $\cG$ such that there is w.h.p. cutoff for the random walk on a random $n$-lift of $\cG$, and we prove that the cutoff window is of order $\sqrt{\log n}$. 
\\
\noindent We introduce the following assumptions:
\begin{enumerate}[label=\textbf{A.\arabic*}]
\item \label{assump1} the RW on $\cG$ is irreducible,
\item \label{assump2}\textbf{[Two-cycles property]} $\cG$ has at least two oriented cycles which are not each other's inverse, where an \textbf{oriented cycle of length $m\geq 1$} is a cyclic order $C=(\e_1, \ldots, \e_m)$ of $m$ oriented edges with positive weight such that the end vertex of $\e_i$ is the initial vertex of $\e_{i+1}$ and $\e_i\neq \e_{i+1}^{-1}$ for all $i\pmod m$, and the \textbf{inverse} of this cycle is the cycle $C^{-1}=(\e_m^{-1}, \ldots, \e_1^{-1})$.
\end{enumerate}

\begin{theorem}\label{thm:bornsup}
Suppose that $\cG$ satisfies \ref{assump1} and \ref{assump2}, and that $\cG_n$ is a uniform random lift of $\cG$. For any $\varepsilon \in (0,1)$,if $t_{max}^{(n)}(\varepsilon)$ is the worst-case mixing time of the RW on $\cG_n$,
\begin{equation}\label{eq:bornsup}
t_{max}^{(n)}(\varepsilon)= h^{-1}\log n +O_{\dP}\left(\sqrt{\log n}\right)
\end{equation}
as $n\rightarrow +\infty$, the constant $h>0$ being the entropy of the universal cover of $\cG$ (see Section \ref{sec:univcover}).
\end{theorem}
\noindent
The following lower bound shows that random $n$-lifts achieve the smallest possible mixing time:

\begin{proposition}\label{prop:borninf}
For any deterministic sequence $(\cG_n)_{n\geq 1}$ of $n$-lifts of $\cG$, and for any $\varepsilon \in (0,1)$, 
\begin{equation}\label{eq:borninf}
\liminf_{n\rightarrow +\infty} t_{min}^{(n)}(\varepsilon)\geq h^{-1}\log n + O_{\dP}(\sqrt{\log n}),
\end{equation}
where $t_{min}^{(n)}(\varepsilon):=\min_{x\in V_n}t_{x}^{(n)}(\varepsilon)$ is the best-case $\varepsilon$-mixing time  (ie, the shortest mixing time among all possible starting vertices).
\end{proposition}

\noindent
Assumption \ref{assump2} is necessary in Theorem \ref{thm:bornsup}: one notices easily that if $\cG$ has at most one oriented cycle, then w.h.p., $\inf_{x\in V_n}\pi_n\left(\{y\in V_n, \text{ there is no oriented path from $x$ to $y$} \}\right)$ is bounded away from 0. 
\\
We conjecture a Gaussian profile for the cutoff window. This can be shown at least for the lower bound (see Section \ref{sec:proofs}). Such a profile was established for the SRW on the random $d$-regular graph $G_n(d)$\cite{cutoffregular}: there is randomness for the speed of the walk (since it can backtrack), but the degree of the vertices met by the walk is constant. Conversely, for the NBRW on the configuration model whose cutoff window also exhibits this behaviour\cite{nbrw}, there is randomness for the degrees, but not for the speed. In our setting, as for the SRW on the configuration model\cite{blps}, both the environment and the speed of the walk might vary, and the result of their combination is not clear. 
\\
Theorem \ref{thm:bornsup} is also true for a \textbf{lazy random walk} on $V_n$ with any \textbf{holding probability} $\alpha\in (0,1)$, i.e. the transition matrix is $P^{(\alpha)}_n(u,v)=\alpha\textbf{1}_{\{u=v\}}+(1-\alpha)\sum_{\e:u\rightarrow v}w(\e)$ for all $u,v\in V_n$ (hence, our RW is lazy with holding probability $1/2$). This gives a new value of the entropy: 
\begin{equation}\label{eq:newentropy}
h_{\alpha}=\frac{h}{2(1-\alpha)}.
\end{equation} 
The question whether this holds in the case $\alpha =0$ remains unsolved, see Appendix 2 (Section \ref{appendix2}) for a discussion.
\\
One might also investigate to what extent Theorem \ref{thm:bornsup} holds when $\cG$ changes with $n$. It is proven in \cite{cutoffregular} that there is still cutoff for the RW on $G_n(d)$ if $d=n^{o(1)}$.

\subsection{Examples}\label{sec:examples}
Our setting is very general, since it includes lifts of any finite Markov chain with positive holding probability. We highlight two special cases below.
\\

\noindent
\textbf{Random walks on the $d$-regular random graph}
\\
We can recover an approximate version of the result of \cite{cutoffregular} for the RW on $d$-regular graphs, when $d$ is even: in the very particular case when $\cG$ consists of a single vertex and $d/2$ loops having weight $1/d$ on both orientations, a random $n$-lift of $\cG$ is a random $d$-regular multigraph (but its distribution is neither that of a uniform $d$-regular multigraph, nor that of $G_n(d)$). Our results allow us to derive the cutoff for the SRW (which is a RW with holding probability $\alpha =0$): Proposition \ref{prop:borninf} is still valuable for $\alpha =0$ and gives the lower bound, and the upper bound comes from Theorem \ref{thm:bornsup} with $\alpha>0$ arbitrarily small. One gets $h_0=\frac{(d-2)\log(d-1)}{d}$ (tools for its computation are in Section \ref{sec:univcover}). This is exactly the value of $h$ in Theorem 1 of \cite{cutoffregular} for the SRW on $G_n(d)$. Their theorem states in addition that the cutoff window is of order $\sqrt{\log(n)}$ with a Gaussian profile, hence corresponding to our lower bound.
\\

\noindent
\textbf{Cutoff for non-Ramanujan graphs}
\\
It was recently proven that on every sequence of $d$-regular weakly Ramanujan graphs, the SRW admits a cutoff \cite{LubetzkyPeres} (a sequence $G_n$ of $d$-regular graphs is said to be \textbf{weakly Ramanujan} whenever for all $\varepsilon >0$, every eigenvalue of the adjacency matrix of $G_n$ is either $\pm d$ or in $[-2\sqrt{d-1}-\varepsilon, 2\sqrt{d-1}+\varepsilon]$ for $n$ large enough). This was even extended to graphs having $n^{o(1)}$ eigenvalues anywhere in $(-d+\varepsilon ', d-\varepsilon')$ for an arbitrary $\varepsilon ' >0$.
\\
Theorem \ref{thm:bornsup} gives an alternative proof for the existence of sequences of non weakly Ramanujan graphs having a cutoff. Indeed, take for $\cG$ a connected $d$-regular graph which is not Ramanujan. One computes easily that all eigenvalues of $\cG$ are also eigenvalues of $\cG_n$, so that $\cG_n$ is not weakly Ramanujan.

\subsection{Tools and reasoning}\label{subsec:tools}
The graph we study has locally few cycles, so that the behaviour of the RW on $\cG_n$ is closely linked to that of a RW on its \textbf{universal cover} $(\cT_{\cG},\circ)$, the infinite rooted tree obtained from $\cG$ by "unfolding" all its non-backtracking paths from a given distinguished vertex, the root (hence, when cycling back to an already visited vertex, treat it as a new vertex, see Figure 2 for an example). A \textbf{non-backtracking path} is an oriented path $(\e_1, \ldots, \e_m)$ such that $\e_{i+1}\neq \e_i^{-1}$ for all $i\leq m-1$. 
\\
We will write abusively $\cT_{\cG}$ instead of $(\cT_{\cG},\circ)$ when the root is irrelevant.
\\
\\
\begin{tikzpicture}[scale=0.5]\label{picture:univcov}

    \node[main node] (1) {$u$};
    \node[main node] (2) [above left = 1cm and 3.5cm of 1]  {$u$};
    \node[main node] (3) [above left = 1cm and 1.5cm of 1]  {$u$};
    \node[main node] (4) [above right = 1cm and 1cm of 1] {$v$};
    \node[main node] (5) [above right = 1cm and 3.5cm of 1] {$x$};
    \node[main node] (6) [above left = 1cm and 0cm of 4] {$x$};%son of v
    \node[main node] (7) [above right = 1cm and 0.5cm of 4] {$x$};%son of v
    \node[main node] (8) [above = 0.7cm of 5] {$v$};%son of x       
    \node[main node] (9) [above right = 1cm and 0.8cm of 5] {$v$};%son of x
    \node[main node] (11) [above left = 1cm and 2.5cm of 2]  {$u$};
    \node[main node] (12) [above left = 1cm and 1cm of 2]  {$v$};
    \node[main node] (13) [above = 0.7cm of 2]   {$x$};
    \node[main node] (14) [above left = 1cm and 0.01cm of 3]  {$u$};
    \node[main node] (15) [above right  = 1cm and -0.2cm of 3]  {$v$};
    \node[main node] (16) [above right = 1cm and 1cm of 3]   {$x$};    
        \node[] (25) [right = 5.9cm of 1] {Level 0};
        \node[] (26) [above right = 1cm and 6.1cm of 1] {Level 1};
        \node[] (27) [above right = 2.7cm and 6.1cm of 1] {Level 2};
\draw [->] (-1.5,0.3) -- (-2.5,0.7);
\draw (-2.1,0.1) node[left] {$0.2$};
\draw [->] (1.5,0.3) -- (2.5,0.7);
\draw (2.1,0.1) node[right] {$0.1$};
\draw [<-] (-6.5,2.3) -- (-7.5,2.7);
\draw (-6.9,2) node[left] {$0.3$};
\draw [<-] (6.5,2.3) -- (7.5,2.7);
\draw (6.9,2) node[right] {$0.3$};
\draw [->] (-10,3.7) -- (-11,4.2);
\draw (-10,3.3) node[left] {$0.2$};
\draw [<-] (-13,5.2) -- (-14,5.7);
\draw (-13.6,4.8) node[left] {$0.3$};
\draw [->] (9.7,3.6) -- (10.5,4.5);
\draw (10,3.6) node[right] {$0.6$};
\draw [<-] (10.7,4.7) -- (11.5,5.6);
\draw (11,5) node[right] {$0.3$};
\draw [->] (-0.9,1) -- (-1.4,1.5);
\draw (-1.1,1.4) node[above] {$0.3$};
\draw [->] (-3.2,2.9) -- (-2.4,2.3);
\draw (-2.7,2.8) node[right] {$0.2$};
\draw [->] (0.5,1) -- (1,1.5);
\draw (0.3,1.2) node[above] {$0.4$};
\draw [->] (2.3,2.8) -- (1.8,2.3);
\draw (2,2.6) node[left] {$0.5$};
  \path[draw,thick]
  (1) edge[blue] node {} (2)
  (1) edge[blue] node {} (3)
  (1) edge[green] node {} (4)
  (1) edge[purple] node {} (5)
  (4) edge[red] (6)
  (4) edge[black] (7)
  (5) edge[red] (8)
  (5) edge[black] (9)
  (2) edge[blue] (11)
  (2) edge[green] (12)
  (2) edge[purple] (13)
  (3) edge[blue] (14)
  (3) edge[green] (15)
  (3) edge[purple] (16)
  ;
\end{tikzpicture}

\begin{center}
Figure 2: the first levels of the universal cover of $\cG$ in Figure 1 (not all weights are on the picture), rooted at $u$. Note that two non-backtracking paths start from $u$, since the blue edge in $\cG$ gives rise to two oriented edges opposite to each other.
\end{center}

\noindent
This object, also called "periodic tree", has been thoroughly studied since the 1990s. We postpone a precise historiography to Section \ref{sec:univcover}. Our main references are an article on trees with finitely man cone types \cite{nagniwoess}, and another on regular languages \cite{Gilch16}.
\\
Let $(\cX_t)_{t\geq 0}$ be a RW on $(\cT_{\cG},\circ)$ starting at the root. A key observation is that this RW is transient, which is intuitive, since $(\cT_{\cG},\circ)$ has an exponential growth by \ref{assump2} (that is, for some $C>1$ and all $R$ large enough, there are more than $C^R$ vertices at distance $R$ from the root), and has a "regular" structure for the RW (\ref{assump1} guarantees the existence of an invariant measure). Thus, we can define its loop-erased trace, or "ray to infinity" $\xi:=(\xi_t)_{t\geq 0}$, $\xi_t$ being the last vertex visited by the RW at distance $t$ of the root.
\\
Let $W_t:=-\log(W(\cX_t))$ where for $x\in (\cT_{\cG},\circ)$, $W(x)=\dP(x\in \xi) $. The following CLT sums up almost all the information we need on $\cT_{\cG}$.

\begin{theorem}[\textbf{CLT for the weight on $\cT_{\cG}$}]\label{thm:weightofray}
There exist $h_{\cT_{\cG}}>0, \sigma_{\cT_{\cG}} \geq 0$, only depending on $\cT_{\cG}$, such that 
\begin{equation}\label{eq:cltweightofray}
\frac{W_t-h_{\cT_{\cG}}t}{\sqrt{t}}\overset{law}{\to}\, \mathcal{N}(0, \sigma^2_{\cT_{\cG}}),
\end{equation}
with the convention that $\mathcal{N}(0,0)=\delta_0$ is the Dirac distribution in $0$.
\end{theorem}

\noindent
The proof relies on the regularity of the structure of $\cT_{\cG}$, that allows us to cut the trajectory of $(\cX_t)_{t\geq 0}$ into i.i.d. excursions. This gives us almost directly a proof for Proposition \ref{prop:borninf}. The proof of Theorem \ref{thm:bornsup} proceeds in three steps:
\begin{enumerate}[label= \alph* )]
\item we couple $(\cX_t)_{t\geq 0}$ to a RW $(X_t)_{t\geq 0}$ on $\cG_n$, imitating the analogous construction in \cite{blps} for the configuration model. This coupling is viable as long as $(X_t)$ does not meet cycles. We can ensure this almost until the mixing time, for most starting points of the RW in $\cG_n$. We stress that the RW on $(\cT_{\cG}, \circ)$ is strongly "localized" around $\xi$:

\begin{proposition}[\textbf{Ray localization}]\label{cor:sticktoray}

$$ \forall R,t\geq 1, \, \dP(\xi \cap B(\cX_t,R)=\emptyset )\leq C_1\exp(-C_2R),$$

\end{proposition}
where for all $y\in (\cT_{\cG},\circ)$ and $r\geq 0$, $B(y,r)$ is the set of vertices $y'$ such that there is an oriented path of length $\leq r$ from $y$ to $y'$. This crucial observation allows us to reveal a limited number of edges while coupling RWs on $(\cT_{\cG},\circ)$ and $\cG_n$, hence reducing the probability to meet a cycle. This leads to an "almost mixing" of the RW on $\cG_n$ after $h^{-1}\log n+O(\sqrt{\log n})$ steps: the mass of $P_n^t(X_0,\cdot)$ is concentrated on values of order $e^{O(\sqrt{\log n})}\pi_n(x)$ for some $t$ such that $t=h^{-1}\log n+O(\sqrt{\log n})$.

\begin{corollary}[\textbf{Almost mixing}]\label{cor:almostmix}
 Let $\delta, \varepsilon,a,K>0$. If $-a,K$ are large enough (depending on $\cG$, $\delta$ and $\varepsilon$), then for $n$ large enough, with probability at least $1-\delta$, $\cG_n$ is such that for all $x\in V_n$,
$$\nu_n(V_n)\geq 1 - \varepsilon,$$
where for all $x'\in V_n$ $\nu_n(x'):=P_n^{t'_n}(x,x')\wedge \frac{\exp(K\sqrt{\log n})}{n}$, and $t'_n:= h^{-1}\log n+a \sqrt{\log n}$. 
\end{corollary}
 
\item As in \cite{blps}, a spectral argument relying on the good expanding properties of random lifts (generalizing a little the result of \cite{liniallifts2}) allows us to make the last jump until the mixing time. We underline the fact that the spectral property holds even if the RW on $\cG$ is not reversible. 

\item Finally, we extend the mixing to every starting point for the RW in $\cG_n$, proving that $(X_t)_{t\geq 0}$ quickly reaches a vertex to which we can apply a). We adapt the technique in \cite{bhlp}, which was originally designed for the configuration model considered in \cite{blps}.
\end{enumerate}

\subsection{Plan}
We start with basic but essential properties in Section \ref{sec:defs}. We study the universal cover of $\cG$ in Section \ref{sec:univcover}, and prove Proposition \ref{prop:borninf} and Theorem \ref{thm:bornsup} in Section \ref{sec:proofs}, under some additional assumptions on $\cG$ introduced in Section \ref{subsec:addassump}. We show in Section \ref{sec:general} that those assumptions on $\cG$ are not necessary. We discuss the computation of the constant $h$ in Section \ref{sec:comput}, and the case $\alpha =0$ in Section \ref{appendix2}.

\section{Basic properties}\label{sec:defs}
\subsection{Three Lemmas}
The next property is an essential tool for building $\cG_n$ while exploring it via a walk on the vertices. 
\begin{lemma}\label{lem:sequential}
A random lift $\cG_n$ of $\cG$ can be generated sequentially as follows. Consider $n$ copies of $\cG$, split every edge $e$ in two half-edges, respectively attached to the first and second vertex of $e$. Define their respective \textbf{type} as the type of the orientation of $e$ starting at the first (resp. second) vertex.
Perform the following operations:
\begin{enumerate}
\item pick a uniform unmatched half-edge, say of type $\e\in \E$,
\item match it with another unmatched half-edge, uniformly chosen among those of type $\e^{-1}$,
\item repeat steps 1. and 2. until all half-edges are matched.
\end{enumerate}
\end{lemma}

\begin{proof}
One checks easily that the obtained structure is a $n$-lift of $\cG$, and that all permutations along edges of $\cG$ are uniform and together independent as in the definition.
\end{proof}

 \noindent
Random walks on lifts admit a natural projection property (whose proof is straightforward):
\begin{lemma}[\textbf{Projection of the lift}]\label{lem:projection}
Fix $n\in \dN$. Let $\cG_n$ be a $n$-lift of $\cG$ and let $(X_t)_{t\geq 0}$ be a RW on $\cG_n$. Let $(\overline{X}_t)_{t\geq 0}$ be the projection of $(X_t)_{t\geq 0}$ on $\cG$, obtained by mapping $X_t$ to its type for all $t\geq 0$.
\\
Then $(\overline{X})_{t\geq 0}$ is a RW on $\cG$.
\end{lemma}
\noindent
Hence, walks on lifts inherit much of the structure of walks on the base graph. 
\\
It is well known that for a Markov chain $(X_t)_{t\geq 0}$ on a finite set $\Omega$ with invariant measure $\pi$, for any $u\in \Omega$, $\lim_{t\rightarrow +\infty} \dP(X_t=u)=\pi(u)$ provided that the chain is aperiodic. We finally state a CLT refining this ergodic property (and not requiring aperiodicity):
\begin{lemma}[\textbf{CLT for Markov chains}]\label{lem:cltmarkov}
Let $f$ be a function from $\Omega$ to $\dR$. For $n\in\dN$, let $S_n:=\sum_{t=0}^{n-1}f(X_t)$. Let $m:=\sum_{u\in \Omega}f(u)\pi(u)$ and $v:=\sum_{u\in \Omega}(f(u)-m)^2\pi(u)$. Then 
$$\frac{S_n-mn}{\sqrt{n}}\rightarrow \mathcal{N}(0,v).$$ 
\end{lemma}
This is a direct application of Theorem 16.1 (p.94) in \cite{markov}.

\subsection{Additional assumptions on $\cG$}\label{subsec:addassump}
We introduce the following additional assumptions on $\cG$:
\begin{enumerate}[label= \textbf{A.3}]
\item \label{assump3} All oriented edges have a positive weight,

\end{enumerate}
\begin{enumerate}[label= \textbf{A.4}]
\item \label{assump4} every oriented edge lies on an oriented cycle.
\end{enumerate}

\noindent
For all $u\in \cG$ and $R\geq 1$, let $\partial B(u,R):=B(u,R)\setminus\partial B(u,R-1)$ be the \textbf{border} of $B(u,R)$. When \ref{assump3} holds on $\cG$, every non-oriented path gives rise to two oriented paths (opposite to each other), and we can define the \textbf{\emph{distance}} $d(x,y)$ between two vertices $x,y$ of the graph as the usual graph distance (that is, the length of the shortest path from $x$ to $y$). 
We state those definitions for $\cG$ but might use them for other graphs. Finally, we denote $\Delta$ the largest degree in $\cG$ and $w_{min}>0$ the smallest positive weight in $\cG$.

\section{Study of the universal cover}\label{sec:univcover}
The goal of this Section is to prove Theorem \ref{thm:weightofray}. Fix an arbitrary vertex $v_*\in V_{\cG}$ and an arbitrary oriented edge $\e_*$ going out of $v_*$ (this choice has no importance for the sequel), and root the universal cover of $\cG$ at $v_*$.

\subsection{Definitions for labelled rooted trees}
\textbf{Label} each vertex $x\in V_{\cT_{\cG}}$ by the vertex $v\in V_{\cG}$ such that the non-backtracking path in $\cG$ starting at $v_*$, and corresponding to the shortest path from $\circ$ to $x$ in $(\cT_{\cG},\circ)$, terminates at $v$ (see Figure 2). Give similarly a \textbf{label} in $E$ (resp. $\E$) to each edge (resp. oriented edge) of $(\cT_{\cG},\circ)$.
In the literature, $(\cT_{\cG},\circ)$ is sometimes called the \textbf{directed cover} of $\cG$, or \textbf{periodic tree} arising from $\cG$. Remark that the universal cover of an irreducible $n$-lift of $\cG$ is also $(\cT_{\cG}, \circ)$.
\\
Due to \ref{assump3}, the RW on $(\cT_{\cG},\circ )$ is an irreducible Markov chain with invariant measure $\widetilde{\pi}$ defined as follows: for all $x\in V_{\cT_{\cG}}$ with label $u$, $\tilde{\pi}(x)=\pi(u)$. Denote $\cPT$ its transition matrix.

\begin{lemma}[\textbf{Projection of the cover}]\label{lem:projectionUnivCov}
If one projects $(\cT_{\cG},\circ)$ on $\cG$ by mapping each vertex, edge and oriented edge to its label, then Lemma \ref{lem:projection} holds: the projection of a RW on $(\cT_{\cG},\circ)$ is a RW on $\cG$.
\end{lemma}
\noindent
The proof is straightforward. Note that  there exists at most $\vert V \vert$ distinct rooted trees $(\cT_{\cG},\circ)$ up to isomorphism. Indeed, two vertices in $V_{\cT_{\cG}}$ with the same label induce the same rooted tree.
\\
For $x\in V_{\cT_{\cG}}$, let $he(x):=d(\circ,x)$ be the \textbf{height} of $x$ in $(\cT_{\cG},\circ)$. The \textbf{subtree $\cT_y$ from $x$ in $(\cT_{\cG},\circ)$} has root $x$, vertex set $V_{x}:=\{y\in V_{\cT_{\cG}}, \text{ $x$ is on any path from $\circ$ to $y$}\}$, and the same edges, weights and labels as $(\cT_{\cG},\circ)$ on $V_{x}$. Note as previously that there exist finitely many such subtrees up to isomorphism. $V_{x}$ is the \textbf{offspring of $x$}. If $y\in V_{x}$, it is a \textbf{descendant} of $x$ \textbf{at generation $he(y)-he(x)$}, and $x$ is an \textbf{ancestor} of $y$. If in addition $he(y)=he(x)+1$, $y$ is a \textbf{child} of $x$ and $x$ is its \textbf{parent}. The \textbf{height-$R$ level of} $\cT_x$ denotes $\partial B(x,R)$, and $B(x,R)\cap V_{x}$ (resp. $\partial B(x,R)\cap V_{x}$) is also called the \textbf{offspring of $x$ up to generation $R$ (resp. offspring of generation $R$, or \textbf{$R$-offspring})}. 

An oriented edge in a rooted tree is \textbf{upward} if the height of its initial vertex is smaller than the height of its end vertex, \textbf{downward} else. Its \textbf{height} is the height of its initial vertex. An oriented path of upward (resp. downward) edges is an \textbf{upward} (resp. \textbf{downward}) \textbf{path}. In a tree, there is at most one edge between two vertices $x,y$. We will denote $\{x,y\}$, $(x,y)$ and $(y,x)$ the corresponding edge and oriented edges.

\subsection{A little history and plan of Section \ref{sec:univcover}}
Lyons \cite{Lyons90} and Takacs \cite{Takacs97} have studied RWs on rooted periodic trees arising from simple graphs, with weights corresponding to the SRW with a positive or negative bias towards the root. Thus, this intersects our setting only when the bias vanishes, the RW being a SRW (hence, it is reversible). The more general case of trees with finitely many cone types has been studied by Nagnibeda and Woess \cite{nagniwoess} in 2002: these rooted weighted trees have finitely many subtrees up to isomorphism (periodic trees obviously have finitely many cone types, while the converse is not true). Their work rely on a fine understanding of the Green function initiated in \cite{Nagni99}, which verifies a finite system of (non-linear) equations, due to the repetitive structure of the tree. Among others, they give a transience criterion for the RW on the tree in terms of the eigenvalues of a matrix associated to the tree, and obtain a CLT for the \textbf{rate of escape} (or \textbf{speed}, i.e. $\lim_{t\rightarrow +\infty} he(\cX_t)/t)$ when it exists) in the transient case.
Similar formulas for Green functions were derived around the same time by Lalley \cite{Lalley00} in the broader setting of regular languages.
Gilch \cite{Gilch16} later gave a formula for the entropy of the RW on regular languages, i.e. a LLN for $\log(P^k(\cX_0,\cX_k))_{k\geq 0}$ where $P$ is the transition matrix of the RW.
\\
We extend slightly some of those results in the setting of periodic trees. We first give in Section \ref{subsec:transience} a simple transience criterion for the RW on the universal cover in terms of the base graph (Proposition \ref{prop:transient}), which we have not found in the literature. In particular, under assumptions \ref{assump1}, \ref{assump2}, \ref{assump3} and \ref{assump4} on $\cG$, the RW on $(\cT_{\cG},\circ)$ is transient (Corollary \ref{cor:transient}). A crucial argument from \cite{gaudilandim} is that the reversibility or not of the RW is irrelevant as soon as there exists an invariant measure (in our case, $\widetilde{\pi}$). Hence, we can apply a classical transience criterion for reversible RWs (see \cite{lyonstransience}), noticing that the size of the ball of radius $R$ grows exponentially with $R$ by \ref{assump2}.
\\
Thus, the loop-erased trace $(\xi_t)_{t\geq 0}$ of the RW is an infinite injective path. It is a Markov chain on $(\cT_{\cG},\circ)$ with an easy description (Proposition \ref{prop:lawofray}). We can then define the \textbf{entropic weight} of a vertex $x\in V_{\cT_{\cG}}$ as the probability $W(x)$ that $x$ is in $\xi$ if the RW starts at $\circ$. The regularity of the structure of $(\cT_{\cG}, \circ)$ gives us a CLT for $\log(W(\xi_t))_{t\geq 0}$ (Corollary \ref{cor:weightofray}). Note that in the analogous Proposition 3 of \cite{blps} (in this case, the graph is locally a Galton-Watson tree), there is only a law of large numbers and a domination for the variance, and this prevents already to determine the profile of the cutoff window.
\\
It remains to prove, in Section \ref{subsubsec:cltunivcov}, that the RW on $(\cT_{\cG},\circ)$ has a positive speed. Precise estimates on the Green function obtained by Lalley \cite{Lalley00} and Nagnibeda and Woess \cite{nagniwoess} show in particular that for fixed $x,y\in \cT_{\cG}$, $\cPT^n(x,y)$ decays exponentially with $n$. We prove that there exists random times $(\tau_i)$ with exponential moments (Proposition \ref{prop:renewal}) such that in $(\cT_{\cG},\circ)$, $he(\cX_{\tau_i})=i$ and $he(\cX_t)>i$ for all $t>\tau_i$. Moreover, the trajectories of the RW before and after $\tau_i$ are independent. Hence we can decompose the RW into i.i.d. excursions between each level with exponential moments. This regularity allows us to prove that $he(\cX_n)_{n\geq 0}$ and $\log(W(\cX_n))_{n\geq 0}$ admit a CLT with nonzero mean (Theorem \ref{thm:weightofray} and Proposition \ref{prop:speed}).

\subsection{Transience of the RW on $\cT_{\cG}$}\label{subsec:transience}
In this section, we give necessary and sufficient conditions on $\cG$ for the RW on $(\cT_{\cG},\circ)$ to be transient, and we prove Proposition \ref{cor:sticktoray}.

\begin{proposition}\label{prop:transient}
Suppose that \ref{assump1} holds for $\cG$. Then the RW on $(\cT_{\cG},\circ)$ is transient if and only if:
\begin{itemize}
\item either $\cG$ verifies Assumption \ref{assump2},
\\
\item or $\cG$ has one oriented cycle $(\e_1, \ldots,\e_m)$ for some $m\geq 1$ such that\\
$w(\e_1)\times\ldots\times w(\e_m) \neq w(\e_m^{-1})\times\ldots \times w(\e_1^{-1})$.
\end{itemize}
\end{proposition}

\begin{proof}
If $\cG$ has no oriented cycle, then $(\cT_{\cG},\circ)$ is finite and isomorphic to $\cG$, and the RW is recurrent.
\\
From now on, assume that $\cG$ has at least one oriented cycle. Lemma \ref{lem:core} below deals with the case where \ref{assump4} does not hold for 
$\cG$. Hence, suppose now the contrary.
\\
If \ref{assump2} does not hold, then by \ref{assump4}, $\cG$ is reduced to a cycle. Denote $C=(\e_1, \ldots, \e_m)$ one orientation of this cycle. $(\cT_{\cG},\circ)$ is a line, and the transition probabilities along this line are periodic and are given by the $w(\e_i)$'s and $w(\e_i^{-1})$'s. One can compute that the probability that a RW on $\cG$ starting at the initial vertex $x$ of $\e_1$ runs through $C$ before running through $C^{-1}$ is 
$$\frac{w(\e_1)\times\ldots\times w(\e_m)}{w(e_1)\times\ldots\times w(\e_m) + w(\e_m^{-1})\times\ldots \times w(\e_1^{-1})}.$$ 
Hence the behaviour of a RW on $(\cT_{\cG},\circ)$ is similar to that of a RW on $\dZ$ with transition probabilities $p(i,i+1)=1-p(i,i-1)=\frac{w(\e_1)\times\ldots\times w(\e_m)}{w(e_1)\times\ldots\times w(\e_m) + w(\e_m^{-1})\times\ldots \times w(\e_1^{-1})}$ for all $i\in \dZ$. In particular, it is recurrent if and only if $w(\e_1)\times\ldots\times w(\e_m) = w(\e_m^{-1})\times\ldots \times w(\e_1^{-1})$. 
For a further study of this one-cycle case, see Woess \cite{Woess85}.
\\
Assume now that \ref{assump2} holds. By Lemma 5.1 in \cite{gaudilandim}, it is enough to prove this Proposition in the case where the RW associated to $(\cT_{\cG},\circ)$ (and hence the RW associated to $\cG$) is reversible. Indeed, this result states that if a discrete Markov chain admits an invariant measure, then it is transient if the additive reversibilization of the chain is. The additive reversibilization of the RW on $(\cT_{\cG},\circ)$ is the RW on $(\cT_{\cG}^*,\circ)$, where $(\cT_{\cG}^*,\circ)$ is obtained from $(\cT_{\cG},\circ)$ by modifying its weights as follows: for all $c,y\in V_{\cT_{\cG}}$, set 
$$w^*(x,y)=\frac{1}{2}\left(w(x,y)+\frac{\tilde{\pi}(y)}{\tilde{\pi}(x)}w(y,x) \right).$$
The RW on $(\cT_{\cG}^*,\circ)$ is an irreducible reversible Markov chain with invariant measure $\pi^*$, and $(\cT_{\cG}^*,\circ)$ is in fact the universal cover of $\cG^*$, obtained from $\cG$ by modifying the weights on its oriented edges in the same fashion.
\\
Now, we can apply a classical transience criterion for discrete reversible Markov chains. The earliest proof we found traces back to 1983 (see \cite{lyonstransience}), and is derived from an analogous theorem of Royden for Riemannian surfaces. It states that if a discrete reversible Markov chain has state space $\Omega$, transition matrix $\textbf{P}$ and invariant measure \tiny{$\Pi$}\normalsize , and if there exists a collection of weights $\nu=(\nu_{i,j})_{i,j\in \Omega}$ such that:
\begin{enumerate}[label=(\roman*)]
\item $\forall i,j\in \Omega,\, \nu_{i,j}=-\nu_{j,i}$,
\item there exists $i_0\in \Omega$ such that for all $i\in \Omega$, $\sum_{j\in \Omega}\nu_{i,j}=\left\{ 
\begin{array}{ll}
1 \text{ if $i=i_0$,}\\ 
0 \text{ else,}
\end{array} 
\right.$
\item $\sum_{i,j\in \Omega} \frac{\nu_{i,j}^2}{\text{\tiny{$\Pi$}\normalsize} (i)\textbf{P}(i,j)}<\infty$,
\end{enumerate}
then the Markov chain is transient (here, we use the convention $0/0 =0$). Looking at conditions $(i)$ and $(ii)$, we can interpret $u$ as a current flow entering at $i_0$ and spreading to infinity through an electrical network. The condition $(iii)$ states that the kinetic energy of the flow is finite.
\\
In our setting, we consider the flow generated by a symmetric RW starting at $\circ$ and moving upwards: for all $x\in (\cT_{\cG}, \circ)$ of height $R\geq 2$, let $x_1$ be $x$'s parent and $x_2$ be $x_1$'s parent, and let $$\nu_{x_1,x}=\frac{\nu_{x_2,x_1}}{deg(x_1)-1},$$ and set $\nu_{x,x_1}=-\nu_{x_1,x}$. For every children $x$ of $\circ$, let $\nu_{\circ,x}=-\nu_{x,\circ}=1/deg(\circ)$. For all $x,y\in V_{\cT_{\cG}}$ such that none is the parent of the other, let $\nu_{x,y}=0$. 
\\
Then $(i)$ holds obviously. $(ii)$ is also straightforward. As for checking $(iii)$, note that $\nu_{x_1,x}$ is the probability that a RW started at $\circ$ moves at each step to a uniform children of its current position. Hence the sum of the transitions from one generation to the next one is always 1, that is, for all $R\geq 0$,
$$\sum_{x\in \partial B(\circ,R), y\in \partial B(\circ,R+1)}\nu_{x,y}=1.$$
\noindent
Remark that there exists $\varepsilon>0$, only depending on $\cG$, such that $\widetilde{\pi}(x)\cPT(x,y)>\varepsilon$ for all neighbours $x,y\in V_{\cT_{\cG}}$. Therefore, 
\begin{align*}
\sum_{x,y\in V_{\cT_{\cG}}} \frac{\nu_{x,y}^2}{\widetilde{\pi}(x)\cPT(x,y)} & <2\varepsilon^{-1}\sum_{R\geq 0}\left(\sum_{(x,y)\in \partial \overrightarrow{B}(\circ,R)}\nu_{x,y}^2\right)\\
\\
 & \leq 2\varepsilon^{-1}\sum_{R\geq 0}\max_{(x,y)\in \partial \overrightarrow{B}(\circ,R)}\nu_{x,y},
\end{align*}
where $\partial \overrightarrow{B}(\circ,R)$ is the set of oriented edges with initial vertex in $\partial B(\circ,R)$ and end vertex in $\partial B(\circ,R+1)$. Hence to check $(iii)$, it is enough to prove that the sequence $(s_R)_{R\geq 1}$ is summable with $s_R:= \max_{(x,y)\in \partial \overrightarrow{B}(\circ,R)}\nu_{x,y}$. 
\\
One checks easily that the irreducibility of the RW on $\cG$, \ref{assump2} and \ref{assump4} together imply that every non-backtracking path on $\cG$, after at most $\vert \E\vert$ steps, meets an oriented edge $\e$ leading to at least two other oriented edges. Hence by Lemma \ref{lem:projectionUnivCov}, for all $x\in (\cT_{\cG},\circ)$, if $R=he(x)$ and $x_1$ is the parent of $x$, the shortest path from $\circ$ to $x_1$ contains at least $\lfloor R-1/\vert \E\vert \rfloor$ such edges $\e$, so that
$$\nu_{x_1,x}\leq K (1/2)^{R/\vert \E\vert}$$
for some positive constant $K$. Hence $s_R\leq K(1-w_{min})^{R/\vert \E\vert}$ for all $R\geq 0$, and the transience is proved. This concludes the proof.
\end{proof}

Suppose now that \ref{assump4} does not hold. $\cG$ can be decomposed into a \textbf{core} $c(\cG)$ satisfying \ref{assump4} and "branches" attached to it. $c(\cG)$ is constructed as follows: erase all vertices $x$ of $\cG$ such that $deg(x)=1$ (call them \textbf{leaves}), and delete the edges attached to $x$. Perform this process again on the resulting multigraph, and so on, until no more vertex is erased. Denote $\cG'$ the graph obtained by this algorithm. For every oriented edge $\e$ of $\cG'$, change its weight to $w(\e)/w'(x)$, where $x$ is the initial vertex of $\e$ and $w'(x)$ is the total weight of the edges in $\cG'$ starting at $x$: this modified graph is $c(\cG)$. Clearly, \ref{assump4} holds on $c(\cG)$. Moreover, the RW associated to $c(\cG)$ is also irreducible (erasing a leaf from $\cG$ does not affect the irreducibility).

\begin{lemma}\label{lem:core}
The RW on $\cT_{\cG}$ is transient if and only if the RW on $\cT_{c(\cG)}$ is.
\end{lemma}

\begin{proof}

Remark that $c(\cT_{\cG})=\cT_{c(\cG)} $, the erased vertices in $\cT_{\cG}$ being exactly those whose labels are erased in $\cG$. Let $(\cX_n)_{n\geq 0}$ be a RW on $\cT_{\cG}$. Note that a.s., $(\cX_n)_{n\geq 0}$ visits infinitely many distinct vertices and edges of $\cT_{c(\cG)}$. The \textbf{trace of }$(\cX_n)$ \textbf{on} $\cT_{c(\cG)}$ is defined as follows: for all $p>m\geq 1$ such that $(\cX_{m-1},\cX_m),(\cX_p,\cX_{p+1})\in \cT_{c(\cG)}$ and $(\cX_i,\cX_{i+1})\not\in \cT_{c(\cG)}$ for all $m\leq i\leq p-1$, erase $\cX_m,\cX_{m+1}, \ldots, \cX_{p-1}$. Denote $(\cX'_n)_{n\geq 0}$ the sequence of remaining $\cX_i$'s (ordered by increasing labels). Remark that the way the weights are chosen in the definition of $c(\cG)$ implies that $(\cX'_n)$ is a RW on $\cT_{c(\cG)}$. Then by Proposition \ref{prop:transient}, $(\cX'_n)$ is transient on $\cT_{c(\cG)}$, and this implies that $(\cX_n)$ is transient on $\cT_{\cG}$.
\end{proof}

\noindent
From now on and until the end of Section \ref{sec:proofs}, assume that $\cG$ satisfies \ref{assump1}, \ref{assump2}, \ref{assump3} and \ref{assump4}.

\begin{corollary}\label{cor:transient}
The RW associated to $(\cT_{\cG},\circ)$ is transient.
\end{corollary}

\noindent
In Section \ref{subsec:tools}, we defined the \textbf{ray to infinity} $\xi$ of a RW $(\cX_t)_{t\geq 0}$ in $(\cT_{\cG},\circ)$ as:
$$\xi:=\{x\in V_{\cT_{\cG}} \,\vert \, \exists s\geq 0, \cX_s=x\, \text{ and } \forall t\geq s, \, \text{$\cX_t\in V_x$} \},$$
\noindent
with $\xi_t:=\xi \cap \partial B(\cX_0,t)$. Due to the regularity of $\cT_{\cG}$, the RW is "uniformly transient" and the probability to make $R$ steps in a given direction decreases exponentially w.r.t. $R$.

\begin{proposition}\label{prop:sticktoposition}
Let $(\cX_t)_{t\geq 0}$ be a RW on $(\cT_{\cG},\circ)$. There exist constants $C_1,C_2 >0$ only depending on $\cG$ such that for all $R\geq 0$, for all $x\not\in B(\circ,R)$, 
$$\dP(\exists t>0, \, X_t=x)\leq C_1\exp(-C_2R).$$
Since those constants are independent of the choice of $\circ$, the Markov Property gives the following generalization: for all $s>0$, $y\in (\cT_{\cG},\circ)$ and $x\not\in B(y,R)$,
$$\dP(\exists t>s, \, X_t=x\vert X_s=y)\leq C_1\exp(-C_2R).$$
\end{proposition}

The proof of this proposition requires the following intermediate result. 

\begin{lemma}\label{lem:allexitistransient}
Let $(\cX_t)_{t\geq 0}$ be a RW starting at $\circ$. Let $x$ be a child of $\circ$. Then:\\
\emph{(1)} $\dP(x\in \xi)\geq \frac{1}{\Delta w_{min}^{\vert \E\vert}}$,
\\
\emph{(2)} there exists a positive constant $C_0$, independent of the choice of $v_*$, such that\\
$\dP(\forall t\geq 1,\,  \cX_t\in V_{x})\geq C_0$.
\end{lemma}

\begin{proof}
There exists an oriented edge $(x_1,x_2)\in \E_{\cT_{\cG}}$ such that a RW started at $x_1$ has a probability $p\geq 1/\Delta$ to begins its ray to infinity by visiting $x_2$. Note that this holds with the same value of $p$ if one replaces $x_1$ (resp. $x_2$) by $x'_1\in V_{\cT_{\cG}}$ (resp. $x'_2\in V_{\cT_{\cG}}$) such that $(x'_1,x'_2)$ has the same label as $(x_1,x_2)$. Let $\e$ be this label. 
\\
One might remark that the irreducibility of the RW on $\cG$, \ref{assump2} and \ref{assump4} together imply that for any $\e_a,\e_b\in \E$, there exists a non-backtracking path $(\e_1, \ldots, \e_m)$ in $\cG$ such that $\e_1=\e_a$ and $\e_m=\e_b$(details left to the reader). Moreover, one might impose that this path is injective, so that $m\leq \vert E\vert$. 
This point in the case where $\e_1$ is the label of $(y_1, \circ)$ and $\e_m=\e$, and the projection property of $(\cT_{\cG},\circ)$ on $\cG$ (Lemma \ref{lem:projectionUnivCov}) imply that there is an upward path $(\e_1', \ldots, \e_m')$ in $\cT_{\cG}$ with $\e_1'=(\circ,x)$ and $\e_m'$ has label $\e$. Hence the Markov Property of the RW on $\cT_{\cG}$ implies that

$$\dP(x\in \xi)\geq \frac{1}{\Delta (w_{min})^m}\geq  \frac{1}{\Delta w_{min}^{\vert \E\vert}},$$
and this concludes the proof of (1).
\\
To prove (2), notice first that the transience of the RW on $(\cT_{\cG}, \circ)$ implies that there exists $x_1\in V_{\cT_{\cG}}$ and $C(\circ)>0$ such that for a RW $(\cX_t)_{t\geq 0}$ started at $\circ$, 
\begin{equation}\label{eq:walkgoesinsubtree}
\dP(\forall t\geq 1, \cX_t\in \cT_{x_1})\geq C(\circ).
\end{equation}
As in the proof of (1), if $(\circ',x_1')$ has the same label as $(\circ,x_1)$, one can replace $\circ$ by $\circ'$ and $x_1$ by $x_1'$ in (\ref{eq:walkgoesinsubtree}) with $C(\circ')=C(\circ)$, hence $C(\circ)$ only depends on the choice of $v_*$. Since $V_{\cG}$ is finite, $C':=\inf_{v_*\in V}C(\circ)>0$. A reasoning similar to the proof of (1) leads to the conclusion.
\\
\end{proof}

\begin{proof}[Proof of Proposition \ref{prop:sticktoposition}]
Let $R \geq 0$ and $x\in V_{\cT_{\cG}}$ be as above. On the shortest path $p$ from $\circ$ to $x$, there are at least $m:=\lfloor (R-2)/\vert \E\vert\rfloor$ vertices (excluding $\circ$ and $x$) that are the initial vertex of two upward edges in the subtree $(\cT_{\cG},\circ)$. Denote $x_1, \ldots x_m$ the first $m$ such vertices by increasing height, and $y_i$ the child of $x_i$ that is on $p$. Let $E_i$ (resp. $E_i'$) be the event that $(\cX_t)_{t\geq 0}$ hits $x_i$ (resp. $y_i$). By Lemma \ref{lem:allexitistransient} (2), after hitting $x_i$, the RW has a probability at least $C_0$ to escape through the child of $x_i$ that is not on $p$, and to never hit $y_i$. Hence by the strong Markov property, $\dP(E'_i\vert E_i)\leq 1-C_0$. Note that $E'_m\subset E_m\subset E'_{m-1}\subset \ldots \subset E'_1\subset E_1$. Therefore,
\begin{align*}
\dP(\exists t>0,\, \cX_t=x)\leq &\dP(E'_m)
\\
\leq &\dP(\cap_{i=1}^m (E_i\cap E'_i))
\\
\leq & \dP(E_1)\times\prod_{i=1}^{m-1}\dP(E_i'\vert E_i) \dP(E_{i+1}\vert E'_i)\times \dP(E'_m\vert E_m)
\\
\leq & (1-C_0)^m.
\end{align*}
The conclusion follows.
\end{proof}

\noindent
As a corollary, we can prove Proposition \ref{cor:sticktoray}:

\begin{proof}[Proof of Proposition \ref{cor:sticktoray}]
Note that $\xi \cap B(\cX_t,R)=\emptyset$ implies that there exists $y\not\in B(\cX_t,R)$ such that for some $s>t$, $\cX_s=y$, so that we can apply the previous Proposition.
\end{proof}

\subsection{CLTs for the RW on the universal cover}\label{subsubsec:cltunivcov}
For $x,y\in V_{\cT_{\cG}}$ such that $w(x,y)>0$, let $\hat{w}(x,y):=\dP(y\in \xi \vert \cX_0=x)$ be the probability that the ray to infinity of a RW $(\cX_t)_{t\geq 0}$ started at $x$ goes through a given neighbour $y$ of $x$. Note that this quantity only depends on the label of $(x,y)$ (denote it $\e$), so that one might define $\hat{w}(\e)=\hat{w}(x,y)$. The second part of Lemma \ref{lem:allexitistransient} ensures that $\hat{w}(\e)>0$ for all $\e\in \E$. Let $\widehat{\cG}$ be $\cG$ with weights $(\hat{w}(\e))_{\e\in \E}$ instead of $(w(\e))_{\e\in \E}$. 
\\
Define the \textbf{Non Backtracking Random Walk (NBRW) on $\widehat{\cG}$} as a RW $(Z_t)_{t\geq 0}$ on $\overrightarrow{E}$, such that for all $\e_1,\e_2\in \E$ and $t\geq 0$,
$$\dP(Z_{t+1}=\e_2\vert Z_t=\e_1)= 
\left\{ 
\begin{array}{lll}
\frac{w(\e_2)}{1-w(\e_1^{-1})} \text{ if the end vertex of $\e_1$ is the initial vertex of $\e_2$ }\\ \hfill\text{and $\e_2\neq \e_1^{-1}$,} \\ 
0 \text{ else.}
\end{array} 
\right.$$

 In particular, by our assumptions on $\cG$, the NBRW associated to $\widehat{\cG}$ is irreducible, so that it has a unique invariant probability measure $\hat{\pi}$.

\begin{proposition}\label{prop:lawofray}{\emph{\textbf{(Theorem F' in \cite{nagniwoess})}}}
Let $(\cX_t)$ be a RW on $\cT_{\cG}$ with $\cX_0=\circ$, and let $\rho_t:=(\xi_t,\xi_{t+1})$ be the $t$-th upward edge of its ray to infinity $\xi$. Then $(\rho_t)_{t\geq 0}$ is a Markov chain, and has the same law as a NBRW on $\cT_{\widehat{\cG}}$.
\end{proposition}

\begin{proof}
It is enough to prove that for all $t_0\geq 0$, for all $x\in V_{\cT_{\cG}}$ of height $t_0+1$, and $x_i$ the vertex of height $i$ in the shortest path from $\circ$ to $x$, $1\leq i\leq t_0$: 
\begin{equation}\label{eq:markovrayinf}
\dP(\rho_{t_0+1} = (x_{t_0},x) \vert\, \forall i\leq t_0,\,\rho_{i} =(x_{i-1},x_i))=\frac{\hat{w}(x_{t_0},x)}{1-\hat{w}(x_{t_0 -1},x_{t_0})}.
\end{equation}
\noindent
Note that $\{ \forall i\leq t_0,\,\rho_{i} =(x_{i-1},x_i) \}= \{\rho_{t_0}=(y_{t_0-1}, y_{t_0}\}$. Let $(\rho^{(\tau)}_t)_{t\geq 0}$ be the ray to infinity of $(\cX_{\tau +t})_{t\geq 0}$, where $\tau=\inf\{t\geq 0, \, X_t=y_{t_0}\}$. (\ref{eq:markovrayinf}) follows from the Strong Markov Property applied to the stopping time $\tau$, and of the equalities $\{\rho_{t_0}=(x_{t_0-1},x_{t_0})\}=\{\tau <+\infty\} \cap \{\rho^{(\tau)}_1\neq x_{t_0-1}\}$ and  $\{\rho_{t_0 +1}=(x_{t_0},x)\}=\{\tau <+\infty\} \cap \{\rho^{(\tau)}_1=x\}$.
\end{proof}
\noindent
We define the \textbf{entropic weight of a vertex $x\in V_{\cT_{\cG}}$} 
as the probability that $x$ is in the ray to infinity of a RW started at $\circ$, and denote it $W(x)$.
Let $(x_0, \ldots, x_{H})$ be the vertices on the shortest path from $\circ$ to $x$, where $H= he(x)$, so that $x_0=\circ$ and $x_{H}=x$. By Proposition \ref{prop:lawofray},
\begin{equation}\label{eq:defofweightvertex}
W(x)=\hat{w}(x_0,x_1)\times \prod_{i=1}^{H-1}\frac{\hat{w}(x_i,x_{i+1})}{1-\hat{w}(x_{i-1},x_i)}.
\end{equation}

\begin{corollary}\label{cor:weightofray}
There exist $h_W>0$, $\sigma_W \geq 0$ such that for all $x\in V_{\cT_{\cG}}$ and $(\cX_t)_{t\geq 0}$ a RW on $(\cT_{\cG},\circ)$ with $\cX_0=\circ$, 
$$\frac{-\log(W(\xi_t))-h_W^{-1}t}{\sqrt{t}} \underset{t\rightarrow +\infty}{\to}\mathcal{N}(0, \sigma_W^2)$$
in distribution. Moreover, $\sigma_W =0$ iff $\cG$ verifies the following cylindrical symmetry: there exists $x\in V_{\cT_{\cG}}$, such that for all $j\geq 1$, all upward (resp. downward) edges in $(\cT_{\cG},x)$ between the levels $j-1$ and $j$ have the same weight.
\end{corollary}

\begin{proof}
From equation (\ref{eq:defofweightvertex}), we have $W(\xi_t)=\frac{\prod_{i=0}^{t-1}\hat{w}(\ell(\rho_i))}{\prod_{i=0}^{t-2} 1-\hat{w}(\ell(\rho_i))}$, so that 
$$-\log(W(\xi_t))=\log(\hat{w}(\ell(\rho_{t-1})))+\sum_{i=0}^{t-2}\left[\log(\hat{w}(\ell(\rho_i)))- \log(1-\hat{w}(\ell(\rho_i)))\right],$$
where $\ell(e)$ is the label of $e$ for all oriented edge $e$. By Proposition \ref{prop:lawofray} and Lemma \ref{lem:projectionUnivCov}, $(\ell(\rho_i))_{i\geq 0}$ is a NBRW on $\widehat{\cG}$. We conclude by Lemma \ref{lem:cltmarkov}. 
\end{proof}

\noindent
It remains now to derive a similar result for $(\log(W(\cX_t)))_{t\geq 0}$, which requires in particular to prove that the RW has a positive speed.
Since $(\cX_t)_{t\geq 0}$ is transient, $\theta_i:=\sup\{t\geq 0 \, \vert\, he(\cX_t)=i\}$ is a.s. well-defined and finite for all $i\in \dN$. Let $\tilde{\theta}_i:=\theta_{i+1}-\theta_i$ be the time between the last visits of the walk at level $i$ and at level $i+1$, for $i\geq 1$. Let $\tilde{\theta}_0:=\theta_1$. We have the following corollary of Theorem 2.5 in \cite{Lalley00} (or Theorem C in \cite{nagniwoess}):

\begin{proposition}\label{prop:greenexpo}
There exists $C_3,C_4$ such that for all $n\geq 0$ and $x\in (\cT_{\cG},\circ)$, 
\begin{equation}\label{eq:greenexpo}
\cPT^n(x,x)\leq C_3\exp(-C_4n).
\end{equation}
\end{proposition}

\noindent
We deduce the following: 
\begin{corollary}\label{cor:expomomentslasttimes}
There exists $C_5>0$ such that for all $i,n\geq 0$, and for all $\e\in \E$ such that $\dP((\cX_{\theta_i-1},\cX_{\theta_i})\text{ has label } \e)>0 $,
\begin{equation}\label{eq:expomomentslasttimes}
\dP(\tilde{\theta}_i\geq n \vert \, (\cX_{\theta_i-1},X_{\theta_i})\text{ has label } \e)\leq C_5\exp(-C_4n).
\end{equation}
\end{corollary}

\begin{proof}
For $i=1$, this is a direct application of the above Proposition \ref{prop:greenexpo}. Now, for $i>1$, note that the law of $(\cX_t)_{t\geq \theta_i}$ is that of a RW started at $\cX_{\theta_i}$ conditioned on making its first step not towards the parent of $\cX_{\theta_i}$, and then never coming back to $\cX_{\theta_i}$ after the start. If $A$ denotes the event that this conditioning happens, then $\dP(A)\geq w_{min}C$ by Lemma \ref{lem:allexitistransient}. Therefore, (\ref{eq:greenexpo}) implies that for all $n\geq 0$,
\begin{align*}
\dP(\tilde{\theta}_i\geq n)\leq &\sum_{m\geq n} \dP(X_{\theta_i +m}=X_{\theta_i})
\\
\leq & \frac{1}{w_{min}C}\sum_{m\geq n}C_3\exp(-C_4m)
\\
\leq & C_5\exp(-C_4n)
\end{align*}
for $C_5$ large enough.
\end{proof}

\noindent
We say that $t\in \dN$ is an \textbf{exit time} if the oriented edge $(\cX_t,\cX_{t+1})$ is upward, has label $\e_*$ and if $\cX_s\neq \cX_t$ for all $s\geq t+1$ (recall that $\e_*\in \E$ was arbitrarily picked at the beginning of Section \ref{sec:univcover}). For $i\geq 0$, let $\tau_i$ be the $i$-th time interval between two exit times and $\e_i$ the corresponding \textbf{exit edge}. Note that the $\e_i$'s are exactly the edges of type $\e$ in $\rho$. Let $\tau_0:=0$ and $\widetilde{\tau}_i:=\tau_{i+1}-\tau_i$ for $i\geq 0$ be the \textbf{$i$-th renewal interval}. Let $\epsilon_i:=(\cX_t)_{\tau_i\leq t \leq \tau_{i+1}-1}$ be the $i$-th \textbf{excursion} between two exit times. Let $x_*$ be the end vertex of the oriented edge of label $\e_*$ starting at $\circ$. We will abusively identify $\epsilon_i$ with its projection on $\overline{(\cT,x_*)}$, a canonical representative of the isomorphic rooted trees $(\cT,z)$, $z\in V_{\cT_{\cG}}$ having the same label as $x_*$. Proposition \ref{prop:lawofray} implies that $\tau_i<\infty $ for all $i\geq 1$ a.s.
\\

\begin{proposition}\label{prop:renewal}
The random variables $(\epsilon_i)_{i\geq 1}$ are i.i.d., and there exists $C_6,C_7>0$ such that for all $m,i\geq 0$, 
\begin{equation}\label{eq:exporenewal}
\dP( \widetilde{\tau}_i\geq m)\leq C_6\exp(-C_7m).
\end{equation}

\end{proposition}

\begin{proof}
Note that for all $i\geq 1$, 
\begin{equation}\label{eq:indepexcursions}
(\cX_{\tau_i +t})_{t\geq 0} \overset{law}{=} (\overline{\cX}_t)_{t\geq 0},
\end{equation}
where $(\overline{\cX}_t)$ is a RW on $(\cT_{\cG}, \circ)$ starting at $x_*$ and conditioned on not leaving  $\overline{(\cT,x_*}$. Hence the variables $\epsilon_i$ all have the same distribution.
\\
As for the independence, note that for all $j\geq 1$, conditionally on $\tau_1, \ldots, \tau_j$ and $\cX_0, \ldots, \cX_{\tau_j}$, the RW $(\cX_t)_{t\geq \tau_j}$ has the law of $(\overline{\cX}_t)_{t\geq 0}$. In particular, $\epsilon_j$ is independent from the sigma-algebra $\sigma\left((\epsilon_0, \ldots, \epsilon_{j-1})\right)$. From this, we deduce that the $\epsilon_i$'s are together independent.
\\
Remark that conditionally on the label of the oriented edge $(\cX_{\theta_{i}-1}, \cX_{\theta_i})$, $\tilde{\theta}_i$ is independent of $\sigma\left((\tilde{\theta_{k}})_{0\leq k \leq i-1}\right)$. By Corollary \ref{cor:expomomentslasttimes}, there exists a probability distribution $\Theta$ on $\dN$ with expectation $\mathcal{E}:=\dE[\Theta]<+\infty$ such that $\dP(\Theta \geq n )\leq C_5\exp(-C_4n)$ for all $n\geq 1$ and such that $\theta_i$ is stochastically dominated by $Z_1+\ldots +Z_i$, where the $Z_j$'s i.i.d. variables of law $\Theta$.
\\
For some constants $K,K'>0$, for all $n\geq 1$, 
\begin{align*}
\dP(\theta_{\lfloor n/2\mathcal{E}\rfloor}\geq n)\leq &\dP(Z_1+\ldots +Z_{\lfloor n/2\mathcal{E}\rfloor} \geq n)
\\
\leq & \frac{\dE[\exp(tZ_1)]^{\lfloor n/2\mathcal{E}\rfloor}}{e^{tn}}
\\
\leq & \left(\frac{\dE[\exp(tZ_1)]}{e^{2\mathcal{E}t}}\right)^{\lfloor n/2\mathcal{E}\rfloor}
\end{align*}
for all $t>0$ by Markov's inequality. Note that for $t<C_4$, $\dE[\exp(tZ_1)]$ is finite, and that for $t\rightarrow 0$, $\dE[\exp(tZ_1)]=\sum_{p\geq 0} t^p\frac{\dE[Z_1^p]}{p!}= 1+\mathcal{E}t+O(t^2)$. Thus, there exists $t_0>0$ such that $ 0<r<1$ where $r:=\dE[\exp(t_0Z_1)]\exp(-2\mathcal{E}t_0)$, and we have $\dP(\theta_{\lfloor n/2\mathcal{E}\rfloor}\geq n)\leq r^{\lfloor n/2\mathcal{E}\rfloor}$.
\\
By Proposition \ref{prop:lawofray}, since the NBRW on $\widehat{\cG}$ is irreducible, it is standard that there exist constants $\alpha, \beta >0$ such that for all $m\geq 0$, 
$$\dP(\tau_2\geq \theta_m)\leq \alpha \exp(-\beta m).$$
Hence, for all $i\geq 0$, since the $(\widetilde{\tau}_i)_{i\geq 1}$ are i.i.d.,  
\begin{align*}
\dP(\widetilde{\tau}_i\geq n) &\leq \dP(\widetilde{\tau}_0 +\widetilde{\tau}_1 \geq n)
\\
&\leq \dP(\theta_{\lfloor n/2\mathcal{E}\rfloor} \geq n) +\dP(\tau_2\geq \theta_{\lfloor n/2\mathcal{E}\rfloor} )
\\
&\leq r^{\lfloor n/2\mathcal{E}\rfloor} + \alpha \exp (-\beta \lfloor n/2\mathcal{E}\rfloor).
\end{align*}
This concludes the proof.
\end{proof}

Recall that $W_t:=-\log(W(\cX_t))$ is the log-weight of the RW. We now prove Theorem \ref{thm:weightofray}.

\begin{proof}[Proof of Theorem 3]
Let $W^{(\epsilon)}_i:=\log(W(\cX_{\tau_i}))-\log(W(\cX_{\tau_{i+1}}))$ for $i\geq 1$. By Proposition \ref{prop:renewal}, the $(\We_i)_{i\geq 1}$'s are i.i.d., and $\dP(\We_1\geq n)\leq \dP(\widetilde{\tau}_1\geq -n/\log(w_{min}))=O(\exp(C'n/\log(w_{min})))$ for all $n\geq 1$, so that $\We_1$ has moments of any order. Let $h_w:=\dE[\We_1]$. Clearly, $\We_1> 0$ a.s., so that $h_w>0$.
\\
Now, cutting the trajectory of the RW into excursions between exit edges, we have
$$W_t=-\log(W(\cX_{\tau_1}))-(\log(W(\cX_t))+\log(W(\cX_{\tau_{r_t}}))+\sum_{i=1}^{r_t} \We_i,$$
where $r_n:=\max\{i\geq 0, \, \tau_i \leq n\}$. 
\\
By Proposition \ref{prop:renewal} again, $\log(W(\cX_{\tau_1}))+\log(W(\cX_t))-\log(W(\cX_{\tau_{r_t}}))=o(\sqrt{t})$ with high probability, so that by Slutsky's Lemma, it is enough to show the existence of $h,\sigma >0$ such that 
$$\frac{\sum_{i=1}^{r_t-1} \We_i-ht}{\sqrt{t}} \overset{law}{\to} \mathcal{N}(0,\sigma^2).$$
Let $\tau:=\dE[\widetilde{\tau}_1]\in (0, \infty)$, $\overline{\We}_i:=\We_i-h_w$ and $\overline{\tau}_i:=\widetilde{\tau}_i-\tau$. For all $\lambda \in \dR$, 
\begin{align*}
\dP\left(\frac{\sum_{i=1}^{r_t-1} \We_i-\frac{h_wt}{\tau}}{\sqrt{t}} \leq \lambda\right)=&\, \dP\left(\sum_{i=1}^{r_t-1} \We_i\leq \frac{h_w t}{\tau}+\lambda\sqrt{t} \right)
\\
=& \dP\left(\frac{h_wr_t+\sum_{i=1}^{r_t-1} \overline{\We}_i}{\tau r_t+\sum_{i=1}^{r_t-1} \overline{\tau}_i} \times \frac{\sum_{i=1}^{r_t-1} \widetilde{\tau}_i}{t} \leq \frac{h_w}{\tau}+\frac{\lambda}{\sqrt{t}} \right)
\\
=& \dP\left(\frac{h_w}{\tau}\left( \frac{1+\sum_{i=1}^{r_t-1} \overline{\We}_i/(h_wr_t)}{1+\sum_{i=1}^{r_t-1} \overline{\tau}_i/(\tau r_t)}\right) \frac{\sum_{i=1}^{r_t-1} \widetilde{\tau}_i}{t}\leq \frac{h_w}{\tau}+\frac{\lambda}{\sqrt{t}} \right).
\end{align*}
The series associated to the sequences $(\We_i)_{i\geq 1}$ and $(\widetilde{\tau}_i)_{i\geq 1}$ either verify a CLT, or are deterministic. Thus, $\sum_{i=1}^{r_t-1} \overline{\We}_i =o(t^{2/3})$ and $\sum_{i=1}^{r_t-1} \overline{\tau}_i =o(t^{2/3})$ with high probability as $t\rightarrow +\infty$. And, the strong law of large numbers implies that $r_t/t\rightarrow 1/\tau$ a.s. Hence, 
$$ \frac{1+\sum_{i=1}^{r_t-1} \overline{\We}_i/(h_wr_t)}{1+\sum_{i=1}^{r_t-1} \overline{\tau}_i/(\tau r_t)}=1+\frac{1}{r_t}\sum_{i=1}^{r_t-1}\left(\frac{\overline{\We}_i}{h_w}-\frac{\overline{\tau}_i}{\tau} \right)+R_t,$$
where $R_t=o(t^{-2/3})$ with high probability. 
\\
Furthermore, $\frac{\sum_{i=1}^{r_t-1} \widetilde{\tau}_i}{t}=1+ \frac{(\tau_{r_t}-t) -\tau_1}{t}=1+o(t^{-2/3})$ with high probability, according to Proposition \ref{prop:renewal}. Hence
$$\frac{h_w}{\tau}\left( \frac{1+\sum_{i=1}^{r_t-1} \overline{\We}_i/(h_wr_t)}{1+\sum_{i=1}^{r_t-1} \overline{\tau}_i/(\tau r_t)}\right) \frac{\sum_{i=1}^{r_t-1} \widetilde{\tau}_i}{t}=
\frac{h_w}{\tau}+ \frac{h_w}{\tau r_t}\sum_{i=1}^{r_t-1}Z_i+R'_t,$$
where $Z_i:=\frac{\overline{\We}_i}{h_w}-\frac{\overline{\tau}_i}{\tau}$ and $R'_t=o(t^{-2/3})$ with high probability, and
\begin{align*}
\dP\left(\frac{\sum_{i=1}^{r_t-1} \We_i-\frac{h_wt}{\tau}}{\sqrt{t}} \leq \lambda\right)=& \dP\left( \frac{1}{r_t}\sum_{i=1}^{r_t-1}Z_i\leq \frac{\tau\lambda}{h_w\sqrt{t}}-R'_t\right)
\\
=& \dP\left(\frac{1}{\sqrt{r_t}}\sum_{i=1}^{r_t-1}Z_i\leq \frac{\tau\lambda}{h_w}\sqrt{\frac{r_t}{t}}-R'_t\sqrt{r_t}\right)
\\
=& \dP\left(\frac{1}{\sqrt{r_t}}\sum_{i=1}^{r_t-1}Z_i\leq \frac{\sqrt{\tau}\lambda}{h_w}+R''_t\right),
\end{align*}
where $R''_t=o(1)$ with high probability. But the $Z_i$'s are i.i.d. variables, and we have $\dE[Z_1]=0$ and $\sigma_Z^2:=Var(Z_1)>0$. Indeed by \ref{assump3}, for a fixed trajectory of $\varepsilon_1$ (hence a given value of $\overline{\We}_1$), $\overline{\tau}_1$ can take different values with positive probability, so that $Z_1$ is not deterministic. And $Z_1$ has exponential moments by Proposition \ref{prop:renewal}, so that $\sigma_Z^2$ is finite. Applying the CLT to the series associated to the sequence $(Z_i)_{i\geq 1}$ concludes the proof, and we have 
$$h_{\cT_{\cG}}=\frac{h_w}{\tau} \text{  and  } \sigma^2_{\cT_{\cG}}=\frac{h_w^2\sigma_Z^2}{\tau }.$$
\end{proof}

\begin{proposition}\label{prop:speed}{\emph{\textbf{(Theorems D and E in \cite{nagniwoess})}}}
There exist $s,\sigma_s >0$ such that 

\begin{equation}\label{eq:speedlln}
\frac{he(\cX_t)}{t}\overset{a.s.}{\to}s
\end{equation}

and

\begin{equation}\label{eq:speed}
\frac{he(\cX_t)-st}{\sqrt{t}}\overset{law}{\to}\mathcal{N}(0,\sigma_s^2).
\end{equation}
\end{proposition}

\begin{proof}
The proof is similar to that of Theorem \ref{thm:weightofray}. Again, the fact that $\sigma_s>0$ is due to \ref{assump3}.
\end{proof}

\begin{remark}\label{rem:unifCLTcover}
The convergences in Theorem \ref{thm:weightofray} and Proposition \ref{prop:speed} do not depend on the choice of $v_*$. Moreover, Theorem \ref{thm:weightofray}, Proposition \ref{prop:speed} and Corollary \ref{cor:weightofray} give
\begin{equation}\label{eq:linkentropies}
h_{\cT_{\cG}}=sh_{W}.
\end{equation}

\end{remark}
We speak about ways of computing $h_{\cT_{\cG}}$ and $s$ in the Appendix 1 (Section \ref{sec:comput}).

\section{Proofs of Proposition \ref{prop:borninf} and Theorem \ref{thm:bornsup}}\label{sec:proofs}

\subsection{The lower bound: Proof of Proposition \ref{prop:borninf}}\label{subsec:lowerbound}
For all $n\geq 1$, let $\cG_n$ be an arbitrary $n$-lift of $\cG$. The proof goes as follows: we couple a RW $(X_t)_{t\geq 0}$ on $\cG_n$ with a RW $(\cX_t)_{t\geq 0}$ on $\cT_{\cG}$. The estimate provided by Theorem \ref{thm:weightofray} on $W_t$ implies that $\cX_t$ is concentrated on $o(n)$ vertices with positive probability for $t$ close to $h^{-1}\log n$, and so is $X_t$. This implies a lower bound on $d_{TV}(P^t(X_0,\cdot),\pi_n)$, since $\pi_n$ is almost uniform on $V_n$.
\\
Fix $x\in V_n$ and $\circ\in V_{\cT_{\cG}}$ such that the label of $\circ$ is the type of $x$. Let $(\cX_t)_{t\geq 0}$ be a RW on $(\cT_{\cG},\circ)$ starting at the root. We couple $(\cX_t)$ with a RW $(X_t)_{t\geq 0}$ on $\cG_n$ in the following manner: let $X_0=x$ a.s. For all $t\geq 0$, $X_{t+1}$ is the unique vertex of $V_n$ such that there is an oriented edge from $X_t$ to $X_{t+1}$ whose type is the label of $(\cX_t,\cX_{t+1})$. Clearly, $(X_t)$ is well defined and is indeed a RW on $\cG_n$. 

We define a map $\phi$ from the set of oriented paths starting at $x$ in $\cG_n$ to that of oriented paths starting at $\circ$ in $\cT_{\cG}$: for all $m\geq 1$ and $p:=(\e_1, \ldots,\e_m)$ an oriented path of length $m$ in $\cG_n$, $\phi(p)=(\e_1', \ldots, \e_m')$ where $\e_1'$ is the unique oriented edge such that its initial vertex is $\circ$, and the label of $\e_1'$ is the type of $\e_1$, and for all $i\geq 2$, $\e_i'$ is the unique edge such that its initial vertex is the end vertex of $\e_{i-1}'$ and the label of $\e_i'$ is the type of $\e_i$. 

\begin{remark}\label{rem:correscovergraph}
For all $y\in V_{\cT_{\cG}}$, if $p_1$ (resp. $p_2$) is an oriented path of length $t\geq 1$ from $\circ$ to $y$, then $\phi^{-1}(p_1)$ and $\phi^{-1}(p_2)$ end at the same vertex of $V_n$. The converse is not true as soon as $\cG_n$ has cycles: for every $x'\in V_n$, there are two distinct non-backtracking paths $p_1$ and $p_2$ from $x$ to $y'$, and the oriented paths $\phi(p_1)$ and $\phi(p_2)$ lead to two different vertices $y'_1$ and $y'_2$ of $(\cT_{\cG},\circ)$. 
\end{remark}

Now, let $\lambda>\lambda' \in \dR$ and define $t_n:=\lfloor \frac{\log n}{h_{\cT_{\cG}}}+\lambda' \frac{\sigma_{\cT_{\cG}}}{h_{\cT_{\cG}}^{3/2}}\sqrt{\log n}\rfloor$. By Theorem \ref{thm:weightofray} and Proposition \ref{prop:speed}, 
$$\liminf_{n\rightarrow +\infty }\dP(\{W_{t_n}\leq h_{\cT_{\cG}}t_n -\lambda\sigma_{\cT_{\cG}} \sqrt{t_n} \} \cap \{\vert he(\cX_{t_n}) - st_n\vert > t_n^{2/3})  \})\geq \Phi(\lambda),$$
where $\Phi(\lambda):=\frac{1}{\sqrt{2\pi}}\int_{\lambda}^{+\infty}e^{-\frac{u^2}{2}}du$. Let $A_n:= \{W_{t_n}\leq h_{\cT_{\cG}}t_n -\lambda\sigma_{\cT_{\cG}} \sqrt{t_n}  \}\cap \{\vert he(\cX_{t_n}) - st_n\vert > t_n^{2/3}  \}$, and define $U_n:=\{y\in V_{\cT_{\cG}}\, \vert \, A_n \cap \{\cX_{t_n}=y\} \neq \emptyset\}$. Note that for all $R>0$, 

$$\sum_{y\in V_{\cT_{\cG}},\\ he(y)=R} \exp(-W(y))=1.$$
Hence for $n$ large enough,
\begin{align*}
\vert U_n \vert\leq &(2t_n^{2/3}+1)\exp(h_{\cT_{\cG}}t_n -\lambda\sigma_{\cT_{\cG}} \sqrt{t_n})
\\
\leq & \frac{3\log n}{h_{\cT_{\cG}}} \exp \left(\log n+(\lambda'-\lambda)\frac{\sigma_{\cT_{\cG}}}{2\sqrt{h_{\cT_{\cG}}}}\sqrt{\log n}\right)
\\
\leq & n \exp \left((\lambda'-\lambda)\frac{\sigma_{\cT_{\cG}}}{4\sqrt{h_{\cT_{\cG}}}}\sqrt{\log n}\right).
\end{align*}
Let $B_n=\{ x'\in V_n\, \vert \, A_n \cap \{ X_n=x'\}\neq \emptyset\}$.
By Remark \ref{rem:correscovergraph}, $\vert B_n\vert \leq \vert U_n\vert$, hence 
\begin{align*}
d_{TV}(P^n(x,\cdot), \pi_n)\geq &\sum_{x'\in B_n}(P^n(X_0,x')-\pi_n(x'))
\\
\geq & \dP(A_n)-K n^{-1}\vert B_n\vert
\\
\geq & \dP(A_n) -K\exp \left((\lambda'-\lambda)\frac{\sigma_{\cT_{\cG}}}{4\sqrt{h_{\cT_{\cG}}}}\sqrt{\log n}\right).
\end{align*}
for some constant $K>0$. Thus, we obtain $\liminf_{n\rightarrow +\infty} d_{TV}(P^n(y,\cdot), \pi_n) \geq \Phi(\lambda).$ Note that this result is uniform in $y$, due to Remark \ref{rem:unifCLTcover}. Therefore,
$$ \liminf_{n\rightarrow +\infty} \inf_{y\in V_n} d_{TV}(P^n(y,\cdot), \pi_n) \geq \Phi(\lambda),$$ 
and this concludes the proof.  We have even proved a more precise statement: for every $\varepsilon\in (0,1)$, 
$$
\liminf_{n\rightarrow +\infty} \frac{t_{min}^{(n)}(\varepsilon)-h^{-1}\log(n)}{\sigma \sqrt{\log(n)}}\geq \Phi^{-1}(\varepsilon),
$$
with $h=h_{\cT_{\cG}}$ and $\sigma= \frac{\sigma_{\cT_{\cG}}}{h_{\cT_{\cG}}^{3/2}}$.

\subsection{The upper bound: Proof of Theorem \ref{thm:bornsup}}\label{subsec:upperbound}
From now on, we focus on the case where $\cG_n$ is a uniform random lift of $\cG$. The proof consists of three parts, as detailed in Section \ref{subsec:tools}.

\subsubsection*{a) Almost mixing for a typical starting point}\label{subsubsec:coupling}
In this paragraph, we prove that with a large probability, $\cG_n$ is such that a RW  $(X_t)_{t\geq 0}$ started at a uniformly chosen vertex $x\in V_n$ has a large probability to stay on some subtree $T$ of $\cG_n$ for $t_n+O(\sqrt{\log n})$ steps. This allows us to couple $(X_t)_{t\geq 0}$ with a RW $(\cX_t)_{t\geq 0}$ on $\cT_{\cG}$. 
\\
\\
\textbf{The construction.} Fix $n\in \dN$, take $x\in V_n$ and $\circ\in V_{\cT_{\cG}}$ such that the type of $x$ is the label of $\circ$.
\\
We reveal the structure of $\cG_n$ edge by edge, starting from $x$ and making use of Lemma \ref{lem:sequential}. At every moment of the exploration, let $\phi$ be a bijection between the vertices of $T$ and those of a subtree $\mathfrak{T}$ of $(\cT_{\cG},\circ)$ such that $\phi(x)=\circ$ and such that there is an edge between $x_1$ and $x_2$ in $T$ if and only if there is an edge between $\phi(x_1)$ and $\phi(x_2)$ in $(\cT_{\cG},\circ)$.
\\
For all $j\geq 0$, at step $j$, let $D_j$ be the height-$j$ level of $T$, and reveal the pairings of the half-edges attached to the vertices in $D_j$. Erase the edges closing a cycle, and place a mark at the corresponding vertices in $D_j\cup D_{j+1}$. Place a mark at vertices $x'\in D_{j+1}$ such that $\phi(x')\not \in  \mathcal{N}(\beta)$, where for all $\delta >0$, $\mathcal{N}(\delta)=\{y\in V_{\cT_{\cG}}, \, W(y)\geq n^{-1}\exp(\delta\sqrt{\log n} )\}$, and $\beta >0$ is chosen arbitrarily. Do not consider those vertices for further steps. 
\\
This construction of $T$ depends on $\beta$. Note by the way that for a fixed representation of $\cG_n$, one does not have necessarily $T(\beta)\subset T(\beta ')$ if $\beta' <\beta$. 
\\
\\
\textbf{The exploration.} Now, fix $\varepsilon\in (0,1)$ and let $g(\varepsilon)=\min(\varepsilon/2, \varepsilon^2)$. For $n\in \dN$, let
\begin{equation}\label{eq:coupling-R}
R:=\lfloor C \log\log n\rfloor
\end{equation}
with $C$ large enough so that $C_1\exp(-C_2R)=o(\log n^{-1})$. Let $(\cX_t)_{t\geq 0}$ be a RW on $(\cT_{\cG},\circ)$ started at the root, and define $X_t:=\phi^{-1}(\cX_t)$ for all $t$ such that $\cX_t\in \phi(T)$. Keep all the notations of Section \ref{sec:univcover} for $(\cX_t)$, in particular, denote $\xi$ its ray to infinity. 
For all $j\geq 1$, let $t_j:=\inf\{t\geq 0, \cX_t\not \in \phi\left(D_1\cup\ldots \cup D_{R+j-1}\right)\}$ and let $x_j$ be the first vertex in $D_{j+R}$ hit by $(X_t)$. Let $\alpha_j$ be the ancestor of height $R$ of $x_j$ and $O_j$ the $R$-offspring of $\alpha_j$. Stop $(X_t)$:

\begin{itemize}
\item at time $t=0$ if $B(x,R+1)$ contains a cycle,

\item at time $t_{j}$ if a cycle is created while matching the half-edges of vertices in $O_j$, or if $(X_t)$ visits a marked vertex or an ancestor of $\alpha_j$ for $t\leq t_{j+1}$. 
\end{itemize}
\noindent
We say that the exploration is $j$-\textbf{successful} whenever $(X_t)$ does not stop until $t_j$.
If the exploration is $j$-successful, then $(X_t)_{0\leq t\leq t_j}$ is a RW on $\cG_n$.
\\
Note that there are two sources of randomness: the matchings of the half-edges in $\cG_n$ and the trajectory of $(X_t)$. Denote $\dP_{ann}$ the annealed probability on $\cG_n$ and $(X_t)$, and $\dP_{\cG_n}$ the quenched probability on $(X_t)$ conditionally on the realization of $\cG_n$. 

\begin{proposition}\label{prop:goodexplo}
Fix $\beta >0$. There exists $a\in \dR$ such that for $n$ large enough,
\\
$\dP_{ann}\left(\exists k\in \dN \, \vert \, t_k\geq h^{-1}\log n+a \sqrt{\log n}\text{ and the exploration is $k$-successful}\right) \geq 1-\varepsilon$.
\end{proposition}

\begin{proof}
For $j\geq 1$, the exploration stops at time $t_j$ only if:
\begin{itemize}
\item a cycle appears while matching the $R+1$-offspring of $\alpha_j$, or
\item $(X_t)$ visits $\alpha_j$ between its respective first visits at $x_j$ and $x_{j+1}$, or 
\item $\phi(\alpha_j)\not \in\xi$, 
or 
\item $\xi_j \not\in \mathcal{N}(\beta)$,
\end{itemize}
and for all $j'\leq j$, none of those conditions is met. Let $E_j:=\{\text{the exploration stops at time } t_j\}$. For $J\in \dN$, on $F_J:=(\cup_{j=0}^{J}E_j)^c$, the exploration is $J$-successful. 
\\
Let $J_0:=\lfloor h^{-1}s\log n+\gamma \sqrt{\log n}\rfloor$ where $\gamma\in \dR$ is such that $\dP(\xi_{J_0}\not\in \cN(\beta))\geq 1-\varepsilon /4$. By Proposition \ref{cor:weightofray} and by (\ref{eq:linkentropies}), this is possible if one takes $\gamma$ small enough. Let us prove that 
\begin{equation}\label{eq:goodexplo}
\dP_{ann}(\cup_{j=0}^{J_0}E_j) \leq \varepsilon /2.
\end{equation}
\noindent
The probability that a cycle arises while revealing $B(x,R+1)$ is at most 
$$\Delta^{R+2}\frac{\Delta^{R+2}}{\mu\vert V\vert n-2\Delta^{R+2}}=O\left((\log n)^{2C\log(\Delta)} /n\right),$$ 
where $\mu$ is the mean degree of a vertex in $\cG$. Indeed, the height-$k$ level of $T$ contains at most $\Delta^k$ vertices, so that we proceed to at most $\Delta +\Delta^2+\ldots +\Delta^{R+1}\leq \Delta^{R+2}$ pairings of half-edges. Hence for each pairing, there remain at least $\mu \vert V\vert n -2\Delta^{R+2}$ unmatched half-edges belonging to vertices not in $T$, and at most $\Delta^{R+2}$ unmatched half-edges belonging to vertices in $T$.
\\
Now, suppose we are on $E_j$ for some $j\geq 1$. $O_j$ contains at most $\Delta^{R}$ vertices. There are at most $\Delta^{R+1}+\Delta n\exp{(-\beta\sqrt{\log n})}$ unpaired half-edges that belong either to vertices in $O_j$ or to already explored vertices that are not yet marked (those vertices are at height $R+j$ in $T$ but not in $O_j$). Hence, the probability of creating a cycle while proceeding to the pairings of this offspring is $O(\Delta^{R}\exp(-\beta\sqrt{\log n}))=o\left((\log n)^{-1}\right).$ Let $E'_j\subset E_j$ be the event that no cycle arises during those matchings. We have $\dP_{ann}(E_j\setminus E'_j)=o\left((\log n)^{-1}\right)$, uniformly in $j$.
\\
Let $E''_j \subset E'_j$ be the event that $(X_t)$ does not visit $\alpha_j$ between its respective first visits at $x_j$ and $x_{j+1}$.
By Proposition \ref{prop:sticktoposition} and our choice for $R$ in (\ref{eq:coupling-R}), $\dP_{ann}(E'_j\setminus E''_j)\leq C_1\exp(-C_2R)=o\left((\log n)^{-1}\right)$. Again, this bound is uniform in $j$. 
\\
Let $E^{(3)}_j\subset E''_j$ be the event that $\phi(\alpha_j)\in\xi$. Again by Proposition \ref{prop:sticktoposition} and (\ref{eq:coupling-R}), we have $\dP_{ann}(E''_j\setminus E^{(3)}_j)\leq C_1\exp(-C_2R)=o\left((\log n)^{-1}\right)$. For $n$ large enough, for all $j \leq J_0$, the choice of $J_0$ implies
$$\dP_{ann}\left(\exists  k\in \{1, \ldots, j, \} \, \xi_k \not\in \mathcal{N}(\beta)\right)\leq \dP_{ann}(\xi_{J_0}\not\in \mathcal{N}(\beta))\leq\varepsilon/4.$$
\\
Hence, if $E^{(4)}_j:=E^{(3)}_j \cap \{\xi_j \in \mathcal{N}(\beta)\} $, 
$$\dP_{ann}\left(\cup_{j=1}^{J_0} (E^{(3)}_j \setminus E^{(4)}_j)\right) \leq \varepsilon /4.$$ Since $E^{(4)}_j \subset E^{(3)}_j$, $(\phi(X_t))_{t\leq t(j)}$ is never more than $R$ steps away from $\xi_j$, hence $(\cX_t)_{t\leq t(j)}$ does not visit a vertex of too small weight on $E^{(4)}_j$. Since $E^{(4)}_j \subset E'_j$, no cycle is created, so that the exploration is not stopped on $E^{(4)}_j$. But $E^{(4)}_j\subset E_j$, hence $E^{(4)}=\emptyset$.
\\
All in all, we obtain that uniformly in $j\in \{1, \ldots, J_0\}$, $\dP_{ann}(E_j\setminus E^{(3)}_j)=o\left((\log n)^{-1}\right)$, so that 
$\dP_{ann}(F_{J_0})\geq 1-\dP_{ann}(\cup_{j=0}^{J_0}E_j) \geq 1-\dP_{ann}(E_0)- \dP_{ann}(\cup_{j=1}^{J_0} E^{(3)}_j)-\sum_{j=1}^{J_0} \dP_{ann}(E_j\setminus E^{(3)}_j)\geq 1- \varepsilon/2$ for $n$ large enough.
\\
Moreover, if $a$ is small enough (w.r.t. the choice of $\gamma$), by Proposition \ref{prop:speed}, we have that for $n$ large enough, $\dP_{ann}\left(t_{J_0}\leq h^{-1}\log n+a\sqrt{\log n}\right)\leq \varepsilon/2$. This concludes the proof.
\end{proof}

\noindent
The result of Proposition \ref{prop:goodexplo} is annealed, and leads to a quenched result for "most" realizations of $\cG_n$, Corollary \ref{cor:almostmix}.

\begin{proof}[Proof of Corollary \ref{cor:almostmix}]
Take $a$ such that Proposition \ref{prop:goodexplo} holds. We have \\
$1-\nu_n(V_n)=\sum_{x'\in V_n}P_n^{t'_n}(x,x')-\nu_n(x')$. Remark that for all $\mathrm{a},\mathrm{b},\mathrm{m}\in \dR$, such that $\mathrm{a}\geq \mathrm{b}$, $\mathrm{a}- \mathrm{a}\wedge \mathrm{m} \leq  (\mathrm{a}-\mathrm{b})+ (\mathrm{b}-\mathrm{b}\wedge
\mathrm{m})$, so that  
$$1-\nu_n(V_n)\leq \hspace{-3.1mm}\sum_{x'\in V_n\setminus T} P_n^{t'_n}(x,x')-\nu_n(x') +\sum_{x'\in T}\left(P_n^{t'_n}(x,x')-P_n^{t'_n}(x,x',T)\right)  + \left(P_n^{t'_n}(x,x',T)- \nu'_n(x')\right)$$
where $P_n^{k}(x,x',T)$ is the probability that a RW on $\cG_n$ started at $x$ reaches $x'$ in $k$ steps without leaving $T$, and $\nu'_n(x'):= P_n^{t'_n}(x,x',T)\wedge \frac{\exp(K\sqrt{\log n})}{n}$, where $T$ is the exploration tree of Proposition \ref{prop:goodexplo}. Applying Markov's inequality to Proposition \ref{prop:goodexplo} implies that with probability at least $1-\sqrt{\varepsilon}$, $\cG_n$ is such that $\dP_{\cG_n}\left( (X_t)_{1\leq t\leq t'_n}\text{ leaves }T\right)\leq \sqrt{\varepsilon}$. For such $\cG_n$,
$$\sum_{x'\in V_n\setminus T} P_n^{t'_n}(x,x')-\nu_n(x') +\sum_{x'\in T}\left(P_n^{t'_n}(x,x')-P_n^{t'_n}(x,x',T)\right) \leq \sqrt{\varepsilon} ,$$
hence $1 -\nu_n(V_n)\leq \sqrt{\varepsilon}+ \sum_{x'\in T}\left(P_n^{t'_n}(x,x',T)- \nu'_n(x')\right).$
Note that for all $x'\in T$, we have $P_n^{t'_n}(x,x',T)=\cPT^{t'_n}(\circ,\phi(x'), \kT)$ where $\cPT^k(\circ,\phi(x'),\kT)$ is the probability that a RW on $(\cT_{\cG},\circ)$ started at $\circ$ reaches $\phi(x')$ in $k$ steps without leaving $\kT$, the subtree of $(\cT_{\cG},\circ)$ corresponding to $T$. Thus, 
\begin{align*}
1 -\nu_n(V_n) &\leq \sqrt{\varepsilon}+ \sum_{x'\in T}\left(TT^{t'_n}(\circ,\phi(x'),\kT)- \nu'_n(x')\right)
\\
& \leq \sqrt{\varepsilon} + \sum_{y'\in \cT_{\cG}}\cPT^{t'_n}(\circ,y')-\left( \cPT^{t'_n}(\circ,y')\wedge \frac{\exp(K\sqrt{\log n})}{n}\right)
\\
&\leq \sqrt{\varepsilon} +1 -\sum_{y'\in \cT_{\cG}} \left( \cPT^{t'_n}(\circ,y')\wedge \frac{\exp(K\sqrt{\log n})}{n}\right).
\end{align*}
\noindent
Let us prove that $\sum_{y'\in \cT_{\cG}} \left( \cPT^{t'_n}(\circ,y')\wedge \frac{\exp(K\sqrt{\log n})}{n}\right) \geq 1 - \sqrt{\varepsilon}$.
To this end, it is enough to establish that $\dP\left(\cPT^{t'_n}(\cX_0,\cX_{t'_n})\geq \exp(K\sqrt{\log n})/n\right)\leq \sqrt{\varepsilon}$, where $\dP$ is the probability associated to $(\cX_t)$. Let $A_n$ be the $R$-ancestor of $\cX_{t'_n}$: we have clearly that $ W(A_n)\leq w_{min}^{-R}W(\cX_{t'_n})$. Hence by Theorem \ref{thm:weightofray}, one might choose $K$ large enough such that for $n$ large enough, one has $\dP\left( W(A_n)\leq \exp(K\sqrt{\log n}/2)/n\right)\geq 1-\sqrt{\varepsilon}$. 
\\
But we have $W(A_n)\geq  \cPT^{t'_n}(\cX_0,\cX_{t_n})\times \dP\left((\cX_t)_{t\geq t'_n}\text{ does not visit }A_n\right)$, and by Proposition \ref{prop:sticktoposition}, $W(A_n)\geq  \left(1-C_1\exp(-C_2R)\right)\cPT^{t'_n}(\cX_0,\cX_{t_n})$. Hence, with probability at least $1-\sqrt{\varepsilon}$, $\cX_{t'_n}$ is such that 
$$\cPT^{t'_n}(\cX_0,\cX_{t'_n})\leq \frac{\exp(K\sqrt{\log n}/2)/n}{1-C_1\exp(-C_2R)}\leq \frac{\exp(K\sqrt{\log n})}{n}.$$
Thus, we have proved that with probability at least $1-\sqrt{\varepsilon}$, $\cG_n$ is such that $\nu_n(V_n)\geq 1-2\sqrt{\varepsilon}$. Since $\varepsilon>0$ is chosen arbitrarily here, this concludes the proof (remark that without loss of generality, we can impose $\varepsilon <\delta^2$, and that $2\sqrt{\varepsilon} \rightarrow 0$ as $\varepsilon \rightarrow 0$).
\end{proof}

\subsubsection*{b) The last jump for mixing}\label{subsubsec:lastjump}
In this section, we complete the mixing initiated in Corollary \ref{cor:almostmix}, proving a weak version of Theorem \ref{thm:bornsup}, for most starting points of the RW on $\cG_n$: for all $\varepsilon >0$, there exists $K(\varepsilon) >0$ such that if $S\subset V_n$ is the set of vertices $s$ verifying $\vert t^{(n)}_{s}(\varepsilon)-h^{-1}\log n\vert \leq K(\varepsilon)\sqrt{\log n}$, then $\pi_n(S)\rightarrow 1$ as $n\rightarrow +\infty$.
\\
We use the fact that w.h.p., $\cG_n$ is an \textbf{expander}:

\begin{proposition}[\textbf{Expansion}]\label{prop:expansion}
Let $L(\cG_n):=\min_{S\subset V_n, \, \pi_n(S)\leq 1/2}\frac{W(S,S^c)}{\pi_n(S)}$ for $n\geq 1$, where for all $A,B\subset V_n$, $W(A,B):=\sum_{x\in A,y\in B,\e: x\rightarrow y}\mu_n(x)w(\e)$ is the total weight from $A$ to $B$. There exists $L>0$ such that w.h.p. as $n\rightarrow +\infty$, 
\begin{equation}\label{eq:expansion}
L(\cG_n) \geq L.
\end{equation}
\end{proposition}
\noindent
We postpone the proof to the end of the section. In the literature, $L(\cG_n)$ is usually called the \textbf{conductance} of $\cG_n$. The largest $L$ such that (\ref{eq:expansion}) holds is the \textbf{\emph{Cheeger constant}} of $(\cG_n)_{n\geq 0}$. An interesting application of this property is the contraction of $L^2$ norms:

\begin{proposition}\label{prop:cheeger}
There exists $\kappa >0$, only depending $\cG$, such that for all $n,t\geq 1$, and all $x\in V_n$,
$$ Var_{\mu}(k^{(n)}_{t,x}-1)\leq Var_{\mu}(k^{(n)}_{t-1}-1)(1-\kappa L(\cG_n)^2),$$ 

\noindent
where $k^{(n)}_{t,x}(y):=\frac{P_n^t(x,x')}{\mu_n}$ for $x,x'\in V_n$ and $t\geq 0$, and $Var_{\mu}(f):=\sum_{x\in V_n}f(x)^2\mu_n(x)$ for $f:V_n\rightarrow \dR$.
\end{proposition}

\noindent
This is a classical property. Arguments for the proof can be found in Section 2.3 of \cite{Montetali}. We stress that it is not necessary for $P_n$ to be reversible. 
\\
It remains to link the total variation distance and the $L^2$ distance. By Corollary \ref{cor:almostmix}, for $n$ large enough and for all $p\geq 0$, with probability at least $1-\sqrt{\varepsilon}$, $\cG_n$ is such that:
$$d_y(t'_n+p)=\Vert P_n^{t'_n+p}(v,\cdot )-\pi_n \Vert_{TV}\leq 2\sqrt{\varepsilon} +\Vert \nu_nP_n^{p}(y,\cdot )-\pi_n \Vert_{TV}\leq 2\sqrt{\varepsilon} +D\Vert \nu_nP_n^{p}(,y\cdot )-\pi_n \Vert_{L^2(\pi_n)},$$ 
for some positive constant $D$. The last inequality is due to Cauchy-Schwarz and the fact that $\pi_{n,min}:= \inf_{u\in V_n}\hat{\pi}_n(u) =\pi_{min}/ n$. By definition of $\nu_n$, $\Vert \nu_n(\cdot )-\pi_n \Vert_{L^2(\pi_n)} \leq \exp(2K\sqrt{\log n}).$
Hence by Propositions \ref{prop:expansion} and \ref{prop:cheeger}, there exists a constant $D'$ such that for $q\geq D'\sqrt{\log n}$,
\begin{equation}\label{eq:cheegerdecrease}
\Vert \nu_nP_n^{q}(,y\cdot )-\pi_n \Vert_{L^2(\pi_n)}\leq \varepsilon,
\end{equation}
so that $d_y(t'_n+q) \leq \varepsilon +2\sqrt{\varepsilon}$. This concludes the proof of this weak version of Theorem \ref{thm:bornsup}.

\begin{proof}[Proof of Proposition \ref{prop:expansion}]
The proof is a corollary from that of Theorem 1 in \cite{liniallifts2}, which states that there exists $\delta >0$ such that w.h.p., $\min_{S\subset V_n, \, \vert S\vert\leq \vert V_n \vert /2}\frac{\vert E(S,S^c)\vert}{\vert S\vert} \geq \delta $, where $E(S,S^c)$ is the set of non-oriented edges with one endpoint in $S$ and one in its complementary. Noticing that $\vert E(S,S^c)\vert =\vert E(S^c, S)\vert$, one might extend this property in the following way: for all $\theta\in (0,1)$, there exists $\delta(\theta)>0$ such that w.h.p., $\min_{S\subset V_n, \, \vert S\vert\leq \theta\vert        V_n \vert}\frac{\vert E(S,S^c)\vert}{\vert S\vert} \geq \delta(\theta) $.
\\
Let now $S\subset V_n$, such that $\pi_n(S)\leq 1/2$. Recall that the invariant distribution of the RW associated to $\cG_n$ is $\pi_n(x)=\pi(u)/n$ for all $x\in V_n$ of type $u$. Then there exists $\theta_0\in (0,1)$ such that for any $S\subset V_n$, if $\pi_n(S)\leq 1/2$, then $\vert S\vert \leq \theta_0\vert V_n\vert$. But $\frac{W(S,S^c)}{\pi_n(S)}\geq \frac{w_{min}\pi_{min}}{\pi_{max}}\frac{\vert E(S,S^c)\vert}{\vert S\vert} $, where $\pi_{min}$ and $\pi_{max}$ are the smallest and largest values taken by $\pi$ on $V_{\cG}$.
\\
This implies (\ref{eq:expansion}), with $L\geq \delta(\theta_0)w_{min}\pi_{min}/\pi_{max}$.
\end{proof}

\subsubsection*{c) Extending the starting point}\label{subsubsec:allstartpts}
Fix $\varepsilon \in (0,1)$. Let $r:=\lfloor C'\log \log n\rfloor$ for some constant $C'>0$ such that $r>2R$, where $R$ was defined in (\ref{eq:coupling-R}). Say that $x\in V_n$ is a \textbf{root} if $B(x,r)$ contains no cycle, and denote $\cR_n$ the set of roots.
\\
This section is organized as follows: first, we prove that with high probability on $\cG_n$, the random walk has a high probability to reach a root in $O(\log \log n)$ steps, uniformly in the starting point (Proposition \ref{prop:rootquickRW}). Second, we prove that with high probability on $\cG_n$, for every root $x$ of $\cG_n$, a RW starting at $x$ has a probability at least $1-3\varepsilon$ to leave $B(x,r)$ at a vertex $y$ from which the exploration described in Section \ref{subsubsec:coupling} a) is $k$-successful for some $k$ such that $t_k\geq t'_n$ (Proposition \ref{prop:rootexit}). This relies on the fact that conditionally on $\{x \in \cR_n\}$:
\begin{itemize}
\item the exploration from a vertex $y\in \partial B(x,r)$ has a probability at most $\varepsilon$ not to be successful (Lemma \ref{lem:partialgoodexplo}), so that the mean value of $p_x$ is at least $1-\varepsilon$,
\item and the explorations from two vertices $y,y'\in \partial B(x,r)$ whose common ancestor in $B(x,r)$ is at distance at least $R$ (which is the case for most such couples $(y,y')$ ) are almost independent, so that $p_x$ should be concentrated around its mean.
\end{itemize}
Third, we give a proof of Theorem \ref{thm:bornsup} (under \ref{assump3} and \ref{assump4}), using the results of Sections \ref{subsubsec:coupling} a) and \ref{subsubsec:lastjump} b).

\begin{lemma}\label{lem:bulb}
W.h.p. as $n\rightarrow +\infty$, for every $x\in V_n$, $B(x,5r)$ contains at most one cycle. 
\end{lemma} 

\begin{proof}
We proceed via a counting argument similar to the beginning of the proof of Proposition \ref{prop:goodexplo}, while estimating the probability that the exploration stops at move $0$: take $x\in V_n$, proceed to the $O(\Delta^{5r})$ successive matchings of half-edges to generate $B(x,5r)$. Each matching has a probability $O(\Delta^{5r}/n)$ to close a cycle. Hence, the probability that at least two matchings close a cycle is $O\left((\Delta^{5r})^2\times (\Delta^{5r}/n)^2\right)=o(1/n)$. By a union bound, 
$\dP(\exists x\in V_n, \, x\not\in \cB_n)=o(1)$. 
\end{proof}

\begin{proposition}\label{prop:rootquickRW}
Let $c:= \frac{3}{2s}$, where $s$ is the constant of Proposition \ref{prop:speed}. For all $\delta >0$, there exists $n_0\geq 1$ such that for all $n\geq n_0$ and for all realizations of $\cG_n$ such that Lemma \ref{lem:bulb} holds,
\begin{equation}\label{eqprop:rootquickRW}
\max_{x\in V_n}P_n^{\lfloor cr\rfloor}(x, V_n\setminus \cR_n)\leq \delta.
%\max_{x\in V_n}P_n^{\lfloor cr\rfloor}(x, V_n\setminus \cR_n)\overset{\dP}{\to} 0.
\end{equation}
As a consequence, $\max_{x\in V_n}P_n^{\lfloor cr\rfloor}(x, V_n\setminus \cR_n)\overset{\dP}{\to} 0$.
\end{proposition}

\begin{proof}
Fix $\delta>0$ and a realization of $\cG_n$ such that Lemma \ref{lem:bulb} holds, for some $n\geq 0$. 
\\
Let $x\in V_n$. If $B(x,5r)$ contains no cycle, then it is isomorphic to $(\cT_{\cG},\circ)$ for some $\circ\in V_{\cT_{\cG}}$ on its first $5r$ levels. For $n$ large enough, for all $\circ\in V_{\cT_{\cG}}$, $\dP(4r/3\leq he(\cX_{\lfloor cr\rfloor})\leq 5r/3)\geq 1-\delta$ if $(\cX_t)_{t\geq0}$ is a RW on $(\cT_{\cG},\circ)$ started at the root. Hence if $(X_t)_{t\geq 0}$ is a RW on $\cG_n$ started at $x$, $\dP(X_{\lfloor cr\rfloor}\in \cR_n)\geq 1-\delta$.
\\
Else, there is only one cycle in $B(x,5r)$. Hence $B(x,5r)$ can be seen as a cycle $C$, with trees rooted on its vertices of degree at least $3$. 
Let $L\in \dN$ be such that for all $\circ\in V_{\cT_{\cG}}$, for all $y\in (\cT_{\cG},\circ)$ such that $he(y)\geq L$, $\dP(\exists t\geq 0, \, \cX_t=x)\leq \delta$ where $(\cX_t)_{t\geq 0}$ is a RW on $(\cT_{\cG},\circ)$ started at the root (for any $\delta$, such $L$ exists, by Proposition \ref{prop:sticktoposition}). 
Let $C(L)$ be the set of vertices at distance at most $L$ of $C$. We claim that for $n$ large enough, $\dP(X_0, X_1, \ldots, X_{\lfloor\log(r)\rfloor }\in C(L))\leq \delta$. Indeed, note that for all $t\leq \log(r),$
\begin{itemize}
\item either $X_t$ is not on $C$, and is at distance less than $L$ from $B(x,5r)\setminus C(x,L)$,

\item or it is on $C$, and it is at distance at most $\vert V\vert$ from a vertex of $C$ that is the root of a tree planted on $C$ (a cycle in $\cG$ cannot contain a path of $\vert V +1\vert$ consecutive vertices of degree less than 3). 
\end{itemize}  
Remark that all trees rooted on $C$ that $X_t$ might visit for $t\leq \log(r)$ are isomorphic to some $(\cT_{\cG},\circ)$ at least on their first $5r-\log(r)$ levels (since we are still in $B(x,5r)$. Therefore, we have $d(X_t, B(x,5r)\setminus C(x,L))\leq \vert V\vert +L$. Hence $\dP(X_{t+\vert V\vert +L}\not \in C(x,L)) \geq w_{min} ^{\vert V\vert + L}$ for all $t\geq 0$. Decomposing $\{1, \lfloor\log(r)\rfloor\}$ into intervals of length $\vert V\vert +L$, we have 
$$\dP(X_1, X_2, \ldots, X_{\lfloor\log(r)\rfloor }\in C(L)) \leq \left(1-w_{min} ^{\vert V\vert + L}\right)^{\lfloor\log(r)/(\vert V\vert +L)\rfloor}\leq \delta$$
for $n$ large enough. \\
Let then $t_0:=\inf\{t\geq 0, X_{t_0}\not\in C(L)\}$. 
By definition of $L$ and Proposition \ref{prop:speed}, for $n$ large enough, with probability at least $1-2\delta$, $(X_t)$ does not visit $C$ for $t_0\leq t\leq t_0+cr$ and $4r/3 \leq d(X_{t_0 +cr},C)\leq 5r/3$. On this event and on $\{t_0\leq \log(r)\}$, $B(X_{\lfloor cr\rfloor}, r)$ contains no cycle so that $X_{\lfloor cr\rfloor}\in \cR_n$. Therefore, on this realization of $\cG_n$, $\dP(X_{\lfloor cr\rfloor}\in \cR_n)\geq 1-3\delta$ uniformly in the starting point $x$.
\\
Hence, we have shown that for all $\delta >0$, for $n$ large enough and $\cG_n$ such that $\cB_n=V_n$,
$$\max_{x\in V_n}P_n^{\lfloor cr\rfloor}(x, V_n\setminus \cR_n)\leq 3\delta \leq 1-\delta,$$
and the conclusion follows by Lemma \ref{lem:bulb}.
\end{proof}
\noindent
For $x\in \cR_n$, let $\dP_x$ be the probability distribution of $\cG_n$ conditionally on the fact that $x$ is a root, let $\lambda_x$ be the hitting measure on $\partial B(x,r)$ of a RW started at $x$, and for all $y\in \partial B(x,r)$, denote $\alpha(y,x)$ the vertex at distance $R$ of $y$ on its shortest path to $x$.
\\
In addition, if at most $n\exp(-\beta\sqrt{\log n}/2)$ edges of $\cG_n$ have been revealed (where $\beta$ is the constant of Proposition \ref{prop:goodexplo}), and if all revealed paths starting from $y$ and leading to revealed cycles have length at least $r$, or go through $\alpha(y,x)$, say that $\cG_n$ is a \textbf{good context for $y$}.
\\
Define the \textbf{cut exploration from $y$} as the exploration performed in \ref{subsubsec:coupling}, except that some matchings may have already been revealed, and give a mark to $\alpha(y,x)$ (hence don't explore the offspring of $\Phi(\alpha)$ and stop the exploration if the RW hits $\alpha(y,x)$). If the cut exploration from $y$ is $k$-successful for some $k$ such that $t(k)\geq t'_n$, we say that this exploration is \textbf{good}.

\begin{proposition}\label{prop:rootexit}
With high probability on $\cG_n$, for all $x\in \cR_n$,
$$\lambda_x(\{y\vert \text{ the cut exploration from $y$ is good}\})\geq 1-3\varepsilon.$$ 
\end{proposition}
\noindent
To prove this statement, notice first that a cut-exploration from a good context has a large probability to be good.

\begin{lemma}\label{lem:partialgoodexplo}
For $n\in \dN$, let $x\in \cR_n$ and $y\in \partial B(x,r)$. Let $\alpha$ be the unique vertex on the shortest path from $x$ to $y$ at distance $R$ of $y$. Suppose that  $\cG_n$ has been partially revealed and is a good context for $y$. Then $\dP\left(\text{the $\alpha$-cut exploration from $y$ is good}\right)\geq 1-\varepsilon$.
\end{lemma}

\begin{proof}
One checks that all the arguments in the proof of Proposition \ref{prop:goodexplo} remain valid. In particular, the probability of creating a cycle while matching the $R+1$ offspring of the $R$-ancestor of the RW in the first $t'_n$ steps is $o(1)$ and on the event that the exploration is not stopped, the RW has a probability $o(1)$ to visit $\alpha$. 
\end{proof}

\begin{proof}[Proof of Proposition \ref{prop:rootexit}]
Let $x$ be a root. Let $\alpha_1, \ldots, \alpha_q$ be the vertices of $\partial B(x,r-R)$, with $q= \vert \partial B(x,r-R)\vert$. For all $i\in \{1, q\}$, let $A_i$ be the set of vertices $y\in \partial B(x,r)$ such that $\alpha_i$ is  on the shortest path from $x$ to $y$. We say that $A_i$ is \textbf{intact} whenever for all $y\in A_i$, for all $j<i$ and $y'\in A_j$, no vertex in the cut exploration from $y'$ is matched to $y$. On the contrary, if such a matching exists, say that the relevant edge is a \textbf{mismatch}.
Denote $V_{\alpha}$ the set of vertices $\alpha_i$ such that $A_i$ is not intact. Denote $V_{mis}$ (resp. $B$) the set of $y\in \partial B(x,r)$ which are not intact (resp. which are intact and whose cut exploration is not good). Clearly, it is enough to prove that 
$$\dP_x(\lambda_x(B\cup V_{mis})\geq 3\varepsilon)=o(1/n),$$
where the $o(1/n)$ is uniform in $x\in \cR_n$.
\\
Let $I_n$ be the cardinality of $V_{\alpha}$, and $J_n$ the number of mismatches. Clearly, $I_n\leq J_n$ a.s. and there exists $K'>0$, independent of $x$, such that for large enough $n$, $\vert\partial B(x,r)\vert \leq \log n^{K'}$. Hence while performing the cut explorations of all $y\in A_i$ for all $i\leq q$, less than $n\exp(-\beta\sqrt{\log n}/2)$ edges are created for $n$ large enough. Edges arise from independent matchings, so that for $n$ large enough, $J_n$ is stochastically dominated by a sum of $n\exp(-\beta\sqrt{\log n}/2)$ independent Bernoulli random variables of parameter $2\log n^{K'}/n$. This entails for all integers $U\geq 1$:
$$\dP_x(I_n \geq U) \leq \dP_x(J_n \geq U) \leq { n\exp(-\beta\sqrt{\log n}/2) \choose U}\left( \frac{2\log n^{K'}}{n}\right)^{U}\leq  \left( \frac{2\log n^{K'}}{\exp(\beta\sqrt{\log n}/2)}\right)^{U}$$
since ${M \choose N}\leq M^N$ for $M,N \in \dN$. Letting $U=\lfloor 3\sqrt{\log n}/\beta \rfloor $, we obtain $\dP_x(I_n \geq U) =o(1/n)$,
and this is uniform in $x\in \cR_n$. But if $I_n \leq\lfloor 3\sqrt{\log n}/\beta \rfloor$, by Proposition \ref{prop:sticktoposition}, 
\\
$\lambda_x(V_{mis})\leq \dP((X_t)\text{ hits a non-intact $\alpha_i$ before leaving $B(x,r)$ })=O\left( \sqrt{\log n} C_2^{r-R}\right)$. For $C'$ large enough in the definition of $r$, $\sqrt{\log n} C_2^{r-R}=o(1)$. 
\\
Therefore, 
$$\dP_x(\lambda_x(V_{mis})\geq \varepsilon)=o(1/n)$$ 
uniformly in $x\in \cR_n$. Hence, it remains to prove that $\dP_x(\lambda_x(B) >2\varepsilon) =o(1/n)$.

\noindent
For $i\leq q$, let $\mathcal{F}_i$ be the $\sigma$-field generated by the cut explorations of the vertices in $\cup_{j=1}^i A_j$, and define $\mathcal{Y}_i:=\lambda_x(B\cap A_i)$, $\mathcal{Z}_i:=\mathcal{Y}_i-\dE_x(\mathcal{Y}_i\vert\, \mathcal{F}_{i-1})$ and $\mathcal{W} _i:=\sum_{j=1}^i Z_j$. According to Lemma \ref{lem:partialgoodexplo}, for all $i\leq q$, 
$$\dE_x(\mathcal{Y}_i\vert\, \mathcal{F}_{i-1})=\mathbf{1}_{\{A_i \text{ is intact}\}}\sum_{y\in A_i}\lambda_x(y) \dP_x(\text{the exploration from $y$ is not good }\vert \mathcal{F}_{i-1})\leq \varepsilon \lambda_x(A_i).$$
In particular, we obtain that $\lambda_x(B)\leq \varepsilon +W_q$. And for all $i\geq 0$,
$$(\mathcal{W} _{i+1}-\mathcal{W} _i)^2 =\mathcal{Z}_{i+1}^2 \leq 2(\mathcal{Y}_i^2 +\dE_x(\mathcal{Y}_i\vert\, \mathcal{F}_{i-1})^2)\leq 2\lambda_x(A_{i+1})^2 +2\varepsilon^2 \lambda_x(A_{i+1})^2\leq 4\lambda_x(A_{i+1})^2,$$
so that 
$$\sum_{i=0}^{q-1} (\mathcal{W} _{i+1}-\mathcal{W} _i)^2\leq 4\sum_{j=1}^q \lambda_x(A_j)^2\leq 4 \max_{1\leq i \leq q}\lambda_x(A_i).$$
Again, by Proposition \ref{prop:sticktoposition}, $\max_{1\leq i\leq q}\lambda_x(A_i)\leq C_2^{r-R}$. We apply Azuma-Hoeffding's inequality to the martingale $(\mathcal{W} _i)_{1\leq i\leq q}$ to get
$$\dP_x(\mathcal{W}_q\geq \varepsilon)\leq \exp\left( -\frac{\varepsilon^2}{2C_2^{r-R}}\right)=o(1/n),$$
so that $\dP_x(\lambda_x(B) >2\varepsilon) =o(1/n)$, the $o(1/n)$ being uniform in all $x\in \cR_n$. This concludes the proof. 
\end{proof}
\noindent
From this and the conclusions of \ref{subsubsec:lastjump} b), we deduce the following:
\begin{corollary}\label{cor:rootexit}
With high probability on $\cG_n$, for all $x\in \cR_n$,
$$\lambda_x\left(\{y\, \vert\, d_y(\lfloor t''_n+D'\sqrt{\log n} \rfloor )\leq \varepsilon+2\sqrt{\varepsilon}\}\right)\geq 1-3\varepsilon,$$
where $D'$ is defined as in (\ref{eq:cheegerdecrease}).
\end{corollary}
\noindent
\begin{proof}[Proof of Theorem \ref{thm:bornsup}]
Let $T_n:=t''_n+D'\sqrt{log(n)}+cr+2s^{-1}r$. For all vertices $x,y\in V_n$, $P_n^{T_n}(x,y)=\sum_{x'\in V_n}P_n^{cr}(x,x')P_n^{T_n-cr}(x',y)$, so that the distance to equilibrium starting from $x$ verifies
$$2d_x(T_n)=\sum_{y\in V_n}\vert\sum_{x'\in V_n}P_n^{cr}(x,x')P_n^{T_n-cr}(x',y) -\pi_n(y)\vert \leq \sum_{x'\in V_n}P_n^{cr}(x,x')\Vert P_n^{T_n-cr}(x',\cdot)- \pi_n\Vert_{_1}$$
by the triangle inequality. Thus by Proposition \ref{prop:rootquickRW}, w.h.p. as $n\rightarrow +\infty$, $\cG_n$ is such that for all $x\in V_n$, 
\begin{equation}\label{eq:mixfromroot}
d_x(T_n)\leq \varepsilon + \sup_{x'\in \mathcal{R}_n}d_{x'}(T_n-cr).
\end{equation}
\\
For all $x'\in \mathcal{R}_n$, again by the triangle inequality,
$$d_{x'}(T_n-cr)\leq \sum_{y\in \partial B(x',r)}\left( \sum_{\ell=0}^{2s^{-1}r}P_n^{\ell}(x',y)d_z(T_n-cr-\ell) \right)+\dP(\tau_e(x')>2s^{-1}r),$$
where $\tau_e(x')$ is the hitting time of $\partial B(x',r)$ by a RW started at $x'$. But $t\mapsto d_y(t)$ is a non-increasing function for all $y\in V_n$, so that 
$$d_{x'}(T_n-cr)\leq \dP(\tau_e(x')>2s^{-1}r)+\sum_{y\in \partial B(x',r)}\lambda_{x'}(z)d_y(t''_n+D'\sqrt{\log n}).$$
\noindent
By Corollary \ref{cor:rootexit}, $\sum_{y\in \partial B(x',r)}\lambda_{x'}(y)d_y(t''_n+D'\sqrt{\log n})\leq 3\varepsilon+(\varepsilon+2\sqrt{\varepsilon})$, and by Proposition \ref{prop:speed}, for $n$ large enough, $\dP(\tau_e(x')>2s^{-1}r)\leq \varepsilon$ for all $x'\in \mathcal{R}_n$. Combining this with (\ref{eq:mixfromroot}), we obtain that w.h.p., $\cG_n$ is such that
$$\sup_{x\in V_n} d_x(T_n)\leq 5\varepsilon +2\sqrt{\varepsilon}.$$
This concludes the proof of Theorem \ref{thm:bornsup}, in the case when \ref{assump3} and \ref{assump4} hold. 
\end{proof}

\section{Relaxing the assumptions}\label{sec:general}
We now establish that assumptions \ref{assump3} and \ref{assump4} are not necessary for Proposition \ref{prop:borninf} and Theorem \ref{thm:bornsup}.

\subsection{Getting rid of \ref{assump3}}\label{subsec:dropA3}
Without loss of generality, we can assume that each edge admits at least one orientation with positive weight. Introduce the weaker assumption
\begin{enumerate}[label=\textbf{A.3*}]
\item \label{assump3*} At least one edge has both orientations with positive weight.
\end{enumerate}
\noindent
Hence, we suppose that at least one edge has exactly one orientation $\e$ with positive weight. In this case, the RW on $(\cT_{\cG}\circ)$ is not irreducible any more. However, due to Lemma \ref{lem:projectionUnivCov}, the RW crosses an oriented edge with label $\e$ after a time $T$ which is stochastically dominated by a geometrical random variable of parameter $p\in (0,1)$ depending on $\cG$. After this crossing, since $\cT_{\cG}$ is a tree, the RW can never come back to its starting point. Hence, it is transient. Moreover, assumptions \ref{assump1} and \ref{assump2} are enough to imply that the RW can reach every isomorphism class of subtrees, independently of the choice of $\circ$, so all constants in Section \ref{sec:univcover} can be made independent of $v_*$ and $\e_*$.
\\
As a consequence, all results of Section \ref{sec:univcover} hold, with the following exceptions:

\begin{itemize}
\item if \ref{assump3*} is not verified, Proposition \ref{prop:speed} becomes $he(\cX_t)=t$ a.s. for all $t$ (hence $s=1$),

\item if \ref{assump3*} is not verified and $\cT_{\cG}$ satisfies the cylindrical symmetry defined in Corollary \ref{cor:weightofray}, Theorem \ref{thm:weightofray} becomes $\sup_{t\in \dN} \vert W_t-h_{\cT_{\cG}}t\vert \leq K_{\cT_{\cG}}$ for some constant $K_{\cT_{\cG}}$ only depending on $\cG$. 
\end{itemize}
This implies straightforwardly Proposition \ref{prop:borninf}.
One checks readily that the reasoning of Sections \ref{subsubsec:coupling} a), \ref{subsubsec:lastjump} b) and \ref{subsubsec:allstartpts} c) is still true, so that Theorem \ref{thm:bornsup} holds.
\\
For Proposition \ref{prop:expansion} in particular, the original proof of Amit and Linial for simple non-directed graphs is based on an argument that only requires \ref{assump1} and \ref{assump2}: pick $V'\subset V_n$, that contains say $k$ vertices of a given type $u$. Those vertices lead to $k$ vertices of type $v$ by irreducibility of the RW associated to $\cG$. Hence if $V'$ has less than $k(1-\varepsilon)$ vertices of type $v$, $W(V',V'^c)/\pi_n(V')\gtrsim \varepsilon$, and we can suppose that every type is represented in $V'$ 'almost' in the same proportion as the others. By \ref{assump2}, there exist two cycles $C_1$ and $C_2\neq C_1^{-1}$ of $\cG$ such that from each of those $k$ vertices of type $u$, we can go along two trajectories featuring the types of $C_1$ and $C_2$ respectively. If we want $W(V',V'^c)/\pi_n(V')\lesssim \varepsilon$, then at least $(1-\varepsilon)k$ of the $k$ $C_1$-like (resp. $C_2$-like) trajectories should end in those $k$ vertices of type $u$. Denote $E_1$ (resp. $E_2$) this event. Remark that we can suppose that $C_1$ and $C_2$ don't lie on the same set of non-oriented edges, so that $E_1$ and $E_2$ are independent, and the probability that $E_1$ happens is the probability that a uniform permutation of $\{1, \ldots, n\}$ sends at least $k(1-\varepsilon)$ elements of $\{1, \ldots, k\}$ in $\{1, \ldots, k\}$. An estimation of this quantity via a union bound on the choice of the $k(1-\varepsilon)$ elements and Stirling's formula, and a union bound over all possible subsets $V'\subset V_n$ of cardinality at most $V_n/2$ finishes the proof.
\\
Note also that one might define the distance between $x$ and $y$ as the length of the shortest non-oriented path between $x$ and $y$ (and define accordingly $B(x,R)$ and $\partial B(x,r)$): the fact that $B(x,R)$ contains more vertices than\\
 $\tilde{B}(x,R):=\{y\vert \, \text{there is an oriented path of length at most $R$ from $x$ to $y$}\}$ does not change anything to our reasoning (since $c^Rb\leq \vert \tilde{B}(x,R)\vert \leq \vert B(x,R)\vert \leq c'^Rb'$ for some fixed $b,b',c,c' >0$ and all $R\geq 1$).

\subsection{Getting rid of \ref{assump4}}
As in Lemma \ref{lem:core}, one can decompose $\cG$ into a core $c(\cG)$ verifying \ref{assump4} and "branches" planted on this core. A similar decomposition holds for $\cG_n$, as a $n$-lift $c(\cG)_n$ of $c(\cG)$ and branches isomorphic to those of $c(\cG)$. By Lemmas \ref{lem:projection} and \ref{lem:cltmarkov}, after $t_0$ steps of a RW $(X_t)_{t\geq 0}$ in $\cG_n$, the number $n(t_0)$ of steps in $c(\cG)_n$ follows a CLT: $(n(t_0)-\mathfrak{a}t_0)/\sqrt{t_0}\rightarrow  \mathcal{N}(0,\mathfrak{b}^2)$ for some constants $\mathfrak{a},\mathfrak{b}$ depending on $\cG$ only. Moreover, the trajectory of $(X_t)$ on $c(\cG)_n$ is a RW associated to $c(\cG)_n$ by the strong Markov property, so that Theorem \ref{thm:bornsup} holds ($h$ is replaced by $h\mathfrak{a}$). As for Proposition \ref{prop:borninf}, the excursion theory presented in Section \ref{subsubsec:cltunivcov} is still true, and so is Proposition \ref{prop:borninf}.

\section{Appendix 1: computing $h=h_{\cT_{\cG}}$ and $s$}\label{sec:comput}
In general, computing exact values for $h_{\cT_{\cG}}$ and $s$ is a difficult problem.
Nagnibeda and Woess \cite{nagniwoess} give two formulas for $s$ depending on Green functions. The first one (equation $(5.8)$) writes
$$s^{-1}=\sum_{\e\in \E} \hat{\pi}(\e)\frac{F_{\e^{-1}}(1)}{w(\e^{-1})(1-F_{\e^{-1}}(1))},$$ 
where $\hat{\pi}$ is the invariant distribution of the NBRW associated to $\widehat{\cG}$, and 
$$F_{\e}(z)=F_{x,y}(z):=\sum_{n\geq 0} \dP(\cX_n=y\text{ and } \not\exists k<n, \, \cX_k=y\vert \cX_0=x)z^n, $$
for all $z\in \dC$ and $x,y\in V_{\cT_{\cG}}$ such that the oriented edge $(x,y)$ exists and has label $\e$, is the "first passage" Green function. Note that the series $F_{\e}$ has a positive radius of convergence.
\\
The functions $(F_{\e})_{\e\in \E}$ verify a non-linear system of equations given by
$$F_{\e^{-1}}(z)= zw(\e^{-1})+z\sum_{\e'\leftarrow \e} w(\e')F_{\e'^{-1}}(z)F_{\e^{-1}}(z)$$
for all $\e\in \E$ (Proposition 2.5 in \cite{nagniwoess}), where $\e'\leftarrow \e$ means that the end vertex of $\e$ is the initial vertex of $e'$. One can establish this by decomposing the trajectory of a RW $(\cX_t)_{t\geq 0}$ as follows: if $u,v$ are such that $(x,y)$ has label $\e$, and if $X_t=y$, then either $X_{t+1}=x$, or $\cX_{t+1}=y'$ for some $y'\in \cT_{\cG}\setminus\{x\}$, which happens with probability $w(\e')$ where $\e'$ is the label of $(y,y')$. Now, $(X_t)$ can come back to $x$ in $k$ steps only by reaching $y$ in $k'\leq k-1$ steps and then reaching $x$ in $k-k'$ steps. 
\\
Letting $q(x,y):=F_{(x,y)}(1)$ the probability that $(\cX_t)$ reaches $y$ at least once if it starts at $x$, the transition matrix $\widehat{Q}$ of the NBRW on $\cT_{\widehat{\cG}}$ verifies 
$$\widehat{Q}(\e,\e')=\frac{w(\e')(1-q(\e'^{-1}))}{\sum_{\e''\leftarrow \e, \\ \e''\neq \e}w(\e'')(1-q(\e''^{-1}))}.$$ 
Indeed, if $(x,y)$ and $(y,y')$ have respective labels $\e$ and $\e'$, if $(X_t)$ starts at $y$, in order to leave to infinity through $y'$:
\begin{itemize}
\item either it goes through $(y,y')$ and never returns back from $w$ to $v$ (this has probability $w(\e')(1-q(\e'^{-1}))$), 
\item or it goes through any $(y,y'')$, comes back to $y$, and will leave to infinity through $y'$ (this has probability $\sum_{\e''\leftarrow \e}w(\e'')q(\e''^{-1})\hat{w}(\e')$.
\end{itemize} 
Recall that $\hat{w}(\e')$ is the probability that $(X_t)$ leaves to infinity through $y$. Hence, we obtain that $\hat{w}(\e')= w(\e')(1-q(\e'^{-1}))+\sum_{\e''\leftarrow \e}w(\e'')q(\e''^{-1})\hat{w}(\e')$, and $\widehat{Q}(\e,\e')=\frac{\hat{w}(\e')}{1-\hat{w}(\e^{-1})}$, from which we derive the formula. Note that $\sum_{\e''\leftarrow \e}w(\e'')q(\e''^{-1})$ since $q(\overrightarrow{g})<1$ for all $\overrightarrow{g}\in \E$. It then remains to compute the unique invariant probability measure of $\widehat{Q}$.
\\
Knowing $(F_{\e}(1))_{\e\in \E}$ and $\hat{\pi}$, one can compute $s$. However, those quantities are the solutions of non-linear systems of equations, for which no explicit general solutions have been found. Even for the seemingly simple case where $\cG$ has only one oriented cycle (and the inverse cycle), continuous fractions are involved to derive non simple expressions for the $F_{\e}$'s in \cite{Woess85}. The second formula for the speed in \cite{nagniwoess} is derived from a powerful theorem in \cite{Sawyer1987}, and involves again those series. For the SRW on periodic trees, Takacs obtains similar equations and gives explicit solutions for three examples (examples 4.7, 4.8 and 4.10 in \cite{Takacs97}). 
\\
$h_{W}$ can be expressed simply in terms of $\widehat{Q}$ and $\hat{\pi}$, namely
\begin{equation}\label{eq:entropyweight}
h_W=\sum_{\e\in \E}\hat{\pi}(\e)(\log(1-\hat{w}(\e))-\log(\hat{w}(\e)),
\end{equation}
hence once one manages to compute $s$, computing $h_{\cT_{\cG}}$ is straightforward (recall that $h_{\cT_{\cG}}=sh_W$).
\\
Gilch \cite{Gilch16} studies transient random walks in the more general context of general languages, and obtains a LLN: 
\begin{equation}\label{eq:entropygilch}
-\log(\cPT^t(\cX_0,\cX_t))/t \overset{a.s.}{\rightarrow} h'
\end{equation}
for some positive $h'$  if the random walk is transient. He proves that $h'$ is an analytic function of the weights in $\cG$ and discusses the possibilities to compute the entropy, and obtains a formula that reduces to (\ref{eq:linkentropies}) in our particular context: $h'$ is given as the product of three factors, $h'=\lambda^{-1}\ell he(Y)$ (Theorem 2.5), where in our particular setting, $\lambda=1$ is the expected distance between $\cX_{\theta_i}$ and $\cX_{\theta_{i+1}}$, $he(Y)=h_{W}$ and $\ell=s$ is the rate of escape, whose computation in \cite{GilchSpeed} is the equivalent of that of \cite{nagniwoess} for regular languages.
\\
Note that in our setting, (\ref{eq:entropygilch}) can be deduced from Kingman's subadditive ergodic theorem: $\cPT^t(\cX_0,\cX_t)\geq\widetilde{ P}^s(\cX_0,\cX_s)\cPT^{t-s}(\cX_s,\cX_t)$ a.s. for all $0\leq s\leq t$. Moreover, if the label of $\cX_0$ is distributed according to $\pi$ (hence, so is the label of $\cX_r$ for all $r\in \dN$), the random variables $\cPT^{t-s}(\cX_s,\cX_t)$ and $\cPT^{t-s}(\cX_0,\cX_{t-s})$ have the same distribution $d_{t-s}$ that only depends on $t-s$, and we can apply Kingman's theorem to obtain (\ref{eq:entropygilch}). Since $\pi$ is strictly positive on all labels and the convergence is a.s., the result must hold if the label of $\cX_0$ is chosen arbitrarily. Note that this does not imply that $h'>0$.

\begin{remark}
Comparing the formula in \cite{Gilch16} and (\ref{eq:linkentropies}), we get that $h'=h_{\cT_{\cG}}$. However, one can come to this equality without using the results of \cite{Gilch16}: the fact that $h_{\cT_{\cG}}\leq h'$ is a consequence from Proposition \ref{cor:sticktoray}. Indeed, for all $\varepsilon>0$ w.h.p. as $t$ goes to infinity, $\cPT^t(\cX_0,\cX_t)\geq \exp(-(h'+\epsilon)t)$, and $(\cX_n)_{n\geq t}$ has a probability $o(1)$ to visit the $\log(\log(t))$-ancestor of $\cX_t$ (the $o(1)$ is uniform conditionally on $\cX_t$, by Proposition \ref{cor:sticktoray}), so that w.h.p., this ancestor and $\cX_t$ itself have a probability at least $\exp(-(h'+2\varepsilon)t)$ to be in the ray to infinity, so that $h_{\cT_{\cG}}\leq h'+2\varepsilon$. 
\\
Conversely, for all $t\geq 0$, there exists a set $S_t$ of at least $\exp((h'-\varepsilon)t)$ vertices, such that for all $x\in S_t$, $\exp(-(h'+\varepsilon)t)\leq \cPT^t(\cX_0,x)\leq \exp(-(h'-\varepsilon)t)$ and $st-\varepsilon t^{2/3}\leq he(x)\leq st+\varepsilon t^{2/3}$, and such that $X_t\in S_t$ w.h.p. The fact that the vertices of $S_t$ are localized in less than $3\varepsilon t^{2/3}$ levels implies that one can write $S_t=\sqcup_{i=1}^m S'_m$ for some $m\geq \vert S_t\vert /\Delta^{4\varepsilon t^{2/3}}\geq \exp((h'-2\varepsilon)t)$ for $t$ large enough, so that for all $i\leq m$, there exists $x_i\in S'_i$ such that:
\\
a) $S'_i\setminus\{x_i\}$ is in the offspring of $x_i$, hence $x_i$ maximizes $W(x)$ for $x\in S'_i$, which implies that $W(x_i)\geq \exp(-(h_{\cT_{\cG}}-\varepsilon)t)$ for all $i$, and that $\sum_{1\leq i \leq m}W(x_i)\geq m \exp(-(h_{\cT_{\cG}}-\varepsilon)t)$, and
\\
b) $x_i$ is in the offspring of no other $x_j$, so that $\sum_{1\leq i \leq m} W(x_i)\leq 1$.
\\
From this, we deduce that $h_{\cT_{\cG}}\geq h'-3\varepsilon$. Further details are left to the reader.
\end{remark}

\section{Appendix 2: is laziness necessary?}\label{appendix2}
It is easily checked that the fact that $\alpha >0$ is not necessary outside the proof of Proposition \ref{prop:cheeger}, so that the rest of our reasoning still holds for $\alpha =0$ with minor changes (for instance, one might have $\sigma=0$ in Proposition \ref{prop:borninf} if $\cG$ does not verify \ref{assump3*} and if $\cT_{\cG}$ has a cylindrical symmetry, and the RW on $\cG$ might have a period $d>1$, in which case one should look at the RW $(X_t)_{t\geq 0}$ at times $\{t=kd+r, \, k\in \dN\}$, for each residue $r$ modulo $d$, details are left to the reader). 
\\
As mentioned in the example of a $d$-regular graph (Section \ref{sec:examples}), we can prove that there is a cutoff for $\alpha =0$ with mixing time $t_{mix} =h_{0}^{-1}\log n=(2h)^{-1}\log n$: the lower bound of Proposition \ref{prop:borninf} holds, and the upper bound comes from Theorem \ref{thm:bornsup}, taking $\alpha \rightarrow 0$, so that $h_{\alpha}\rightarrow h_0=h/2$ according to (\ref{eq:newentropy}). However, we do not know any more if the cutoff window is still of order $\sqrt{\log n}$.
\\
A sufficient condition to guarantee that the results of \cite{Montetali} required for the proof of Proposition \ref{prop:cheeger} hold would be that there exists $c>0$ such that for all $n$ large enough, 
\begin{equation}\label{eq:spectral}
\inf_{f, Var_{\pi_n}(f)=1}\mathcal{E}_{P_n^*P_n}(f,f) \geq c\inf_{f,Var_{\pi_n}(f)=1} \mathcal{E}_{P_n}(f,f)
\end{equation}
where $P_n^*(x,y):=\frac{\pi_n(y)}{\pi_n(x)}P_n(y,x) $ and $\mathcal{E}_{P_n}(f,f):=\frac{1}{2}\sum_{x,y\in V_n}\left(f(x)-f(y)\right)^2P_n(x,y)\pi_n(x)$ and $\mathcal{E}_{P_n^*P_n}$ is defined analogously (note that $\pi_n$ is invariant for $P_n$ and $P_n^*P_n$).
\\
This is clearly true for $\alpha >0$: for all $x,y\in V_n$, $P_n^*P_n(x,y)\geq P_n^*(x,y)P_n(x,x)\geq \alpha P_n^*(x,y)$, hence $\mathcal{E}_{P_n^*P_n}(f,f)\geq \alpha \mathcal{E}_{P_n}(f,f)$ for all $f$ (note that $\mathcal{E}_{P_n}(f,f)=\mathcal{E}_{P_n^*}(f,f)$ for all $f$), and we are done.

\bibliographystyle{plain}
\bibliography{cutofflifts}
\end{document}